\documentclass[11pt]{article}
\usepackage{CJK}
\usepackage{latexsym,bm}
\usepackage{xypic}
\usepackage{amsmath}
\usepackage{wasysym}
\usepackage{indentfirst}
\usepackage{amssymb}
\usepackage{dsfont}
\usepackage{amsthm}
\begin{document}
\newcommand{\fr}[2]{\frac{\;#1\;}{\;#2\;}}
\newtheorem{theorem}{Theorem}[section]
\newtheorem{lemma}{Lemma}[section]
\newtheorem{proposition}{Proposition}[section]
\newtheorem{corollary}{Corollary}[section]
\newtheorem{conjecture}{Conjecture}[section]
\newtheorem{remark}{Remark}[section]
\newtheorem{definition}{Definition}[section]
\newtheorem{example}{Example}[section]
\newtheorem{notation}{Notation}[section]
\numberwithin{equation}{section}
\newcommand{\Aut}{\mathrm{Aut}\,}
\newcommand{\CSupp}{\mathrm{CSupp}\,}
\newcommand{\Supp}{\mathrm{Supp}\,}
\newcommand{\rank}{\mathrm{rank}\,}
\newcommand{\col}{\mathrm{col}\,}
\newcommand{\len}{\mathrm{len}\,}
\newcommand{\leftlen}{\mathrm{leftlen}\,}
\newcommand{\rightlen}{\mathrm{rightlen}\,}
\newcommand{\length}{\mathrm{length}\,}
\newcommand{\wt}{\mathrm{wt}\,}
\newcommand{\diff}{\mathrm{diff}\,}
\newcommand{\lcm}{\mathrm{lcm}\,}
\newcommand{\dom}{\mathrm{dom}\,}
\newcommand{\fun}{\mathrm{fun}\,}
\newcommand{\SUPP}{\mathrm{SUPP}\,}
\newcommand{\supp}{\mathrm{supp}\,}
\newcommand{\End}{\mathrm{End}\,}
\newcommand{\Hom}{\mathrm{Hom}\,}
\newcommand{\ran}{\mathrm{ran}\,}
\newcommand{\Mat}{\mathrm{Mat}\,}
\newcommand{\rk}{\mathrm{rk}\,}
\newcommand{\rs}{\mathrm{rs}\,}
\newcommand{\piv}{\mathrm{piv}\,}
\newcommand{\perm}{\mathrm{perm}\,}
\newcommand{\inv}{\mathrm{inv}\,}
\newcommand{\orb}{\mathrm{orb}\,}
\newcommand{\id}{\mathrm{id}\,}
\newcommand{\soc}{\mathrm{soc}\,}
\newcommand{\unit}{\mathrm{unit}\,}
\newcommand{\word}{\mathrm{word}\,}

\title{A Gr\"{o}bner Basis Approach to Combinatorial Nullstellensatz}
\author{Yang Xu$^1$ \,\,\,\,\,\, Haibin Kan$^2$\,\,\,\,\,\,Guangyue Han$^3$}
\maketitle

\renewcommand{\thefootnote}{\fnsymbol{footnote}}

\footnotetext{\hspace*{-12mm} \begin{tabular}{@{}r@{}p{13.4cm}@{}}
$^1$ &Department of Mathematics, Faculty of Science, The University of Hong Kong, Pokfulam Road, Hong Kong, China. {E-mail:12110180008@fudan.edu.cn}\\
$^2$ & Shanghai Key Laboratory of Intelligent Information Processing, School of Computer Science, Fudan University,
Shanghai 200433, China.\\
&Shanghai Engineering Research Center of Blockchain, Shanghai 200433, China.\\
&Yiwu Research Institute of Fudan University, Yiwu City, Zhejiang 322000, China. {E-mail:hbkan@fudan.edu.cn} \\
$^3$ & Department of Mathematics, Faculty of Science, The University of Hong Kong, Pokfulam Road, Hong Kong, China. {E-mail:ghan@hku.hk} \\

\end{tabular}}

\vskip 3mm

{\hspace*{-6mm}\bf Abstract---}\! In this paper, using some conditions that arise naturally in Alon's combinatorial Nullstellensatz as well as its various extensions and generalizations, we characterize Gr\"{o}bner bases consisting of monic polynomials, which helps us to establish a Nullstellensatz from a Gr\"{o}bner basis perspective. As corollaries of this general Nullstellensatz, we establish four special Nullstellensatz, which, among others, include a common generalization of the Nullstellensatz for multisets established in K\'{o}s, R\'{o}nyai and M\'{e}sz\'{a}ros \cite{23,24} and the Nullstellensatz with multiplicity established in Ball and Serra \cite{9}, and include a punctured Nullstellensatz, generalizing several existing results in the literature.  As applications of our punctured Nullstellensatz, we extend some results on hyperplane covering in \cite{9,23,24} to wider settings, and give an alternative proof of the generalized Alon-F\"{u}redi theorem established in Bishnoi, Clark, Potukuchi and Schmitt \cite{12}. Unless specified otherwise, all our results are established over an arbitrary commutative ring $R$.

\section{Introduction}
Throughout the paper, we let $\mathbb{Z}^{+}$ denote the set of all the positive integers, and let $\mathbb{N}=\mathbb{Z}^{+}\cup\{0\}$.

Let $\mathbb{F}$ be a field, $n\in\mathbb{Z}^{+}$, $S_1,\dots,S_n$ be nonempty finite subsets of $\mathbb{F}$, and set $g_i=\prod_{u\in S_i}(x_i-u)\in\mathbb{F}[x_1,\dots,x_n]$ for $i=1,\dots,n$. The following two celebrated theorems, known as combinatorial Nullstellensatz, have been established in Alon \cite{3}.

\begin{theorem}([3, Theorem 1.1])
Let $f\in\mathbb{F}[x_1,\dots,x_n]$ with $f(a_1,\dots,a_n)=0$ for all $a\in\prod_{i=1}^{n}S_i$. Then, there exist $h_1,\dots,h_n\in\mathbb{F}[x_1,\dots,x_n]$ such that $f=\sum_{i=1}^{n}h_i\cdot g_i$ and
\begin{equation}\mbox{$\deg(h_i)+\deg(g_i)\leqslant\deg(f)$ for all $i=1,\dots,n$}.\end{equation}
Moreover, if $f,g_1,\dots,g_n\in R[x_1,\dots,x_n]$ for some subring $R$ of $\mathbb{F}$, then $h_1,\dots,h_n$ above can be chosen from $R[x_1,\dots,x_n]$.
\end{theorem}

\setlength{\parindent}{0em}
\begin{theorem}([3, Theorem 1.2])
Let $f\in\mathbb{F}[x_1,\dots,x_n]$, and let $(t_1,\dots,t_n)\in\mathbb{N}^{n}$ such that $\deg(f)=\sum_{i=1}^{n}t_i$, and the coefficient of $\prod_{i=1}^{n}{x_i}^{t_i}$ in $f$ is nonzero. Suppose that $t_i\leqslant|S_i|-1$ for all $i=1,\dots,n$. Then, there exists $a\in\prod_{i=1}^{n}S_i$ such that $f(a_1,\dots,a_n)\neq0$.
\end{theorem}

\setlength{\parindent}{2em}
Combinatorial Nullstellensatz is a very powerful algebraic tool, and has numerous applications in combinatorics, graph theory, combinatorial number theory and coding theory; see, among many others, the papers \cite{2,3,5,30,32,35} and the recent book \cite{36} for more details.

As detailed below, Alon's combinatorial Nullstellensatz has been generalized and extended in a number of directions.

Ball and Serra prove in \cite{9} a Nullstellensatz with multiplicity, which, with the multiplicities of zeros set to be $1$, boils down to Alon's combinatorial Nullstellensatz. They also establish a punctured Nullstellensatz, which is used to recover the Alon-F\"{u}redi theorem on the number of zeros of a polynomial (see \cite{4}), and to generalize many results in Alon and F\"{u}redi \cite{4} and Bruen \cite{15}. K\'{o}s and R\'{o}nyai prove in \cite{24} a Nullstellensatz for multisets, which, with multisets assumed to be ordinary sets, boils down to Alon's combinatorial Nullstellensatz. As in \cite{9}, they establish a punctured Nullstellensatz, which is used to generalize many applications of Alon's combinatorial Nullstellensatz to multisets, including the theorem of Alon and F\"{u}redi on hyperplane covering (see \cite{4}), the Cauchy-Davenport theorem, Sun's theorem on value sets (see \cite{35}) and the Eliahou-Kervaire theorem (see \cite{20}). Mezei proves in \cite{29} a more general Nullstellensatz which includes both the aforementioned Nullstellensatz as special cases.

Laso\'{n} proves in \cite{27} a generalization of Theorem 1.2 by weakening the assumption that $\deg(f)=\sum_{i=1}^{n}t_i$ to that $(t_1,\dots,t_n)$ is maximal among all $(c_1,\dots,c_n)\in\mathbb{N}^{n}$ where the coefficient of $\prod_{i=1}^{n}{x_i}^{c_i}$ in $f$ is nonzero. Batzaya and Bayarmagnai further weaken the assumption for $(t_1,\dots,t_n)$ and prove in \cite{11} a common generalization of Laso\'{n}'s result and [24, Theorem 6].

Now, suppose that $\mathbb{F}$ is replaced by an arbitrary commutative ring $R$. Schauz proves in \cite{34} that if for any $i=1,\dots,n$, $S_i$ satisfies Condition (D), i.e., $v-u$ is not a zero divisor of $R$ for all $u\neq v\in S_i$ (see \cite{12,17,18,34} or Definition 2.2 for more details), then the conclusion of Theorem 1.1 remains valid. Micha{\l}ek proves in \cite{30} that if all the $S_i$'s satisfy Condition (D), then the conclusion of Theorem 1.2 remains valid. Clark proves in \cite{17} that Condition (D) is in fact a necessary and sufficient condition for Theorem 1.1 to hold true. Clark proves in \cite{18} a variant of the punctured Nullstellensatz in Ball and Serra \cite{9} under the assumption that all the $S_i$'s satisfy Condition (D). It is also observed in \cite{18} that the Nullstellensatz in \cite{9} and its punctured version can be generalized under Condition (D). Kulosman and Wang prove in \cite{25} a Nullstellensatz for multisets under the assumption that $R$ is an integral domain of characteristic zero. K\'{o}s, M\'{e}sz\'{a}ros and R\'{o}nyai prove in \cite{23} a Nullstellensatz for multisets under Condition (D), which generalizes both Micha{\l}ek's result in \cite{30} and the Nullstellensatz for multisets in \cite{24}.

We refer the reader to \cite{17,28,31} for some other extensions and generalizations of Theorems 1.1 and 1.2.

As in \cite{23,24,28,29,32}, Gr\"{o}bner bases (see \cite{1,16,21}) arise naturally in Alon's combinatorial Nullstellensatz and its generalizations. In this paper, we re-examine Alon's combinatorial Nullstellensatz and some of its generalizations from a Gr\"{o}bner basis perspective. Our starting point is the observation that each of Theorem 1.1, Theorem 1.2 and Laso\'{n}'s generalization of Theorem 1.2 on maximality of $(t_1,\dots,t_n)$ is in fact a necessary and sufficient condition for $(g_1,\dots,g_n)$ to be a Gr\"{o}bner basis; and moreover, such a fact extends to more general settings including the Nullstellensatz with multiplicity and the Nullstellensatz for multisets.

In Section 2, we collect some notations, definitions and lemmas. In Section 3, we present alternative characterizations of Gr\"{o}bner bases consisting of monic polynomials (Definition 2.1, Theorem 3.1). Three of these characterizations have appeared in Alon's combinatorial Nullstellensatz and its generalizations, yet there is also one characterization that seems to be new.

In Section 4, we apply Theorem 3.1 to derive a Nullstellensatz (Theorem 4.2), where Condition (D) plays a critical role (also see Theorem 4.1). Theorems 4.1 and 4.2 are general in the sense that they can be used to derive more specific Nullstellensatz, as detailed in Section 5.

In Section 5.1, we apply Theorems 4.1 and 4.2 to establish an extension of Mezei's Nullstellensatz [29, Theorem 6.12] (Theorem 5.1). In Section 5.2, we further apply Theorem 5.1 to establish a common generalization of the Nullstellensatz with multiplicity in \cite{9} and the Nullstellensatz for multisets in \cite{23,24} (Theorem 5.2). In Section 5.3, with the help of Theorem 5.2, we establish a punctured Nullstellensatz (Theorem 5.3), which includes the punctured Nullstellensatz in \cite{9,24} as special cases. In Section 5.4, we apply Theorems 4.1 and 4.2 to establish another Nullstellensatz (Theorem 5.4), which is further used to generalize a result in Sauermann and Wigderson \cite{33} (Corollary 5.2). We note that in each of Theorems 5.1, 5.2 and 5.4, Condition (D) arises as a necessary and sufficient condition with some mild assumptions.

In Section 6.1, we apply our punctured Nullstellensatz to hyperplane covering (Theorem 6.1, Corollaries 6.1 and 6.2), generalizing some related results in \cite{9,23,24}. In Section 6.2, following the spirit of \cite{9,18}, we apply our punctured Nullstellensatz to give an alternative proof of the generalized Alon-F\"{u}redi theorem established in Bishnoi, Clark, Potukuchi and Schmitt \cite{12}, a result that has many applications in combinatorics and coding theory (see \cite{12} for more details).

\section{Preliminaries}

\setlength{\parindent}{2em}
First, we introduce some notations and terminologies that will be used throughout the remainder of the paper. Fix $n\in\mathbb{Z}^{+}$, and let
\begin{equation}[1,n]=\{1,\dots,n\}.\end{equation}
Let $\mathbf{0}$ denote the all-zero vector of $\mathbb{N}^{n}$. For any $\alpha\in\mathbb{N}^{n}$ and $i\in[1,n]$, we let $\alpha_i$ denote the $i$-th entry of $\alpha$. For $\alpha,\beta\in\mathbb{N}^{n}$, we define $\alpha+\beta\in\mathbb{N}^{n}$ and $\alpha-\beta\in\mathbb{Z}^{n}$ as
\begin{equation}\mbox{$(\alpha+\beta)_i=\alpha_i+\beta_i$ for all $i\in[1,n]$},\end{equation}
\begin{equation}\mbox{$(\alpha-\beta)_i=\alpha_i-\beta_i$ for all $i\in[1,n]$},\end{equation}
and we write $\alpha\leqslant\beta$ if
$\alpha_i\leqslant\beta_i$ for all $i\in[1,n]$. For any $A\subseteq\mathbb{N}^{n}$, we define
\begin{equation}\Delta(A)=\{\beta\in\mathbb{N}^{n}\mid\exists~\alpha\in A~s.t.~\beta\leqslant\alpha\},\end{equation}
\begin{equation}\nabla(A)=\{\beta\in\mathbb{N}^{n}\mid\exists~\alpha\in A~s.t.~\alpha\leqslant\beta\},\end{equation}
\begin{equation}\max(A)=\{\beta\in A\mid\forall~\alpha\in A:\beta\leqslant\alpha\Longrightarrow\beta=\alpha\}.\end{equation}
For any $A\subseteq\mathbb{N}^{n}$ and $\theta\in A$, $\theta$ is referred to as \textit{the greatest element} of $A$ if $\gamma\leqslant\theta$ for all $\gamma\in A$. For any $A,B\subseteq\mathbb{N}^{n}$, we let
\begin{equation}A+B=\{\alpha+\beta\mid\alpha\in A,\beta\in B\}.\end{equation}
Let $R$ be a commutative ring with multiplicative identity $1_{R}$, and let
\begin{equation}\Omega\triangleq R[x_1,\dots,x_n]\end{equation}
denote the polynomial ring over $R$ in $n$ variables $x_1,\dots,x_n$. For $f\in\Omega$ and $\alpha\in\mathbb{N}^{n}$, we let $f_{[\alpha]}$ denote the coefficient of $\prod_{i=1}^{n}{x_i}^{\alpha_i}$ in $f$. For any $f\in\Omega$, we let
\begin{equation}\supp(f)=\{\alpha\in\mathbb{N}^{n}\mid f_{[\alpha]}\neq0\}.\end{equation}
For any $A\subseteq\mathbb{N}^{n}$, let
\begin{equation}\delta(A)=\{f\in\Omega\mid \supp(f)\subseteq A\}.\end{equation}
We also define $\Phi:(R^{n}\times\mathbb{N}^{n})\times\mathbb{N}^{n}\longrightarrow R$ as
\begin{equation}\hspace*{-12mm}\Phi((u,\alpha),\gamma)=\begin{cases}
\left(\prod_{k=1}^{n}\binom{\gamma_k}{\alpha_k}\right)\left(\prod_{k=1}^{n}{u_k}^{\gamma_{k}-\alpha_{k}}\right),&\alpha\leqslant\gamma;\\
0,&\alpha\nleqslant\gamma.
\end{cases}
\end{equation}
Moreover, for any $f,g\in\Omega$, we write $g\mid f$ if there exists $h\in\Omega$ with $f=h\cdot g$; for any $U\subseteq\Omega$, we let $\langle U\rangle$ denote the ideal of $\Omega$ generated by $U$; and for any finitely generated $R$-module $M$, we let $\rank_{R}(M)$ denote the minimum number of generators of $M$.

\setlength{\parindent}{2em}
Next, we give some definitions. We begin by introducing {\bf{monic polynomials}}.

\setlength{\parindent}{0em}
\begin{definition}
For any $g\in\Omega$, $g$ is said to be monic if there exists $\theta\in\supp(g)$ such that $g_{[\theta]}=1_{R}$ and $\gamma\leqslant\theta$ for all $\gamma\in\supp(g)$, that is to say, $\theta$ is the greatest element of $\supp(g)$ and $g_{[\theta]}=1_{R}$.
\end{definition}

\setlength{\parindent}{2em}
Now following [17, Definition 1] and [34, Definition 2.8], we state the aforementioned Condition (D).

\setlength{\parindent}{0em}
\begin{definition}
For any $u\in R$, $u$ is referred to as a multiplicative unit of $R$ if there exists $v\in R$ with $uv=1_{R}$, and $u$ is referred to as a zero divisor of $R$ if there exists $w\in R$ such that $w\neq0$ and $uw=0$. For $X\subseteq R$, we say that $X$ satisfies Condition (F) in $R$ if for any $a,b\in X$ with $a\neq b$, $b-a$ is a multiplicative unit of $R$; and we say that $X$ satisfies Condition (D) in $R$ if for any $a,b\in X$ with $a\neq b$, $b-a$ is {\textbf{not}} a zero divisor of $R$.
\end{definition}

\setlength{\parindent}{2em}
Now we state our definition of Gr\"{o}bner basis. {\bf{In this paper, we only consider Gr\"{o}bner bases consisting of monic polynomials}}. We refer the reader to [1, Definition 4.1.13] for the general definition of Gr\"{o}bner basis.

\setlength{\parindent}{0em}
\begin{definition}
Let $\Lambda$ be a finite set, and let $(g(\lambda)\mid\lambda\in\Lambda)\in\Omega^{\Lambda}$ be a family of monic polynomials. For any $\lambda\in\Lambda$, let $\theta(\lambda)$ be the greatest element of $\supp(g(\lambda))$. Let $Q$ be an ideal of $\Omega$ with $\{g(\lambda)\mid\lambda\in\Lambda\}\subseteq Q$. Then, we say that $(g(\lambda)\mid\lambda\in\Lambda)$ is a Gr\"{o}bner basis of $Q$ if for any $\tau\in Q-\{0\}$, there exists $\alpha\in\supp(\tau)$ and $r\in\Lambda$ with $\theta(r)\leqslant\alpha$.
\end{definition}

\setlength{\parindent}{2em}
We give some remarks on Definition 2.3.

\setlength{\parindent}{0em}
\begin{remark}
First, a Gr\"{o}bner basis is usually defined as a set of polynomials instead of a tuple. In this paper, we find it convenient to handle a tuple since we do not assume that $g(\lambda)\neq g(\mu)$ for all $\lambda\neq\mu\in\Lambda$.

\hspace*{4mm}\,\,Second, in general, a Gr\"{o}bner basis is defined with respect to a fixed monomial order on $\mathbb{N}^{n}$, and the $g(\lambda)$'s do not need to be monic. Following [21, Sections 21.2 and 21.3], a monomial order on $\mathbb{N}^{n}$ is a total order $\preccurlyeq$ on $\mathbb{N}^{n}$ satisfying the following two conditions:

$(i)$\,\,For any $\alpha,\beta,\gamma\in\mathbb{N}^{n}$ with $\alpha\preccurlyeq\beta$, it holds that $\alpha+\gamma\preccurlyeq\beta+\gamma$;

$(ii)$\,\,For any $\alpha,\beta\in\mathbb{N}^{n}$ with $\alpha\leqslant\beta$, it holds that $\alpha\preccurlyeq\beta$.

Now let $g\in\Omega$ be monic, and let $\theta$ be the greatest element of $\supp(g)$. Then, for an arbitrary term order $\preccurlyeq$, it follows from $(ii)$ that $\alpha\preccurlyeq\theta$ for all $\alpha\in\supp(g)$, that is to say, $\prod_{i=1}^{n}{x_i}^{\theta_i}$ is the leading monomial of $g$ with respect to $\preccurlyeq$; and moreover, the leading coefficient of $g$ with respect to $\preccurlyeq$ is equal to $g_{[\theta]}=1_R$ (see [21, Definition 21.7]). Based on this observation, Definition 2.3 is indeed a special case of the general definition of Gr\"{o}bner basis. Moreover, if $R$ is a field, then Definition 2.3 becomes a special case of [1, Definition 1.6.1], [16, Section 2.8] and [21, Definition 21.25].
\end{remark}

\setlength{\parindent}{2em}
Now we follow \cite{23,24} and introduce the notion of multiset, which will be used in Sections 5 and 6.

\setlength{\parindent}{0em}
\begin{definition}
A pair $(X,\psi)$ is referred to as a multiset if $X$ is a set, $\psi$ is a function with $X\subseteq\dom(\psi)$, and $\psi(u)\in\mathbb{Z}^{+}$ for all $u\in X$.
\end{definition}

\setlength{\parindent}{2em}
The following definition of $r$-hyperplane, which will be used in Section 6, extends that of hyperplane.

\setlength{\parindent}{0em}
\begin{definition}
For any $H\subseteq R^{n}$, $r\in\mathbb{N}$, $H$ is referred to as an $r$-hyperplane of $R^{n}$ if there exists $h\in\Omega$ such that $H=\{u\in R^{n}\mid h(u_1,\dots,u_n)=0\}$ and $\deg(h)=r$. A $1$-hyperplane of $R^{n}$ is referred to as a hyperplane of $R^{n}$.
\end{definition}

\setlength{\parindent}{2em}
Now, we give some lemmas. The following lemma collects some basic properties of monic polynomials that will be used frequently in our discussion. For the sake of completeness, a proof is included in Appendix A.

\setlength{\parindent}{0em}
\begin{lemma}
Let $g\in\Omega$ be a monic polynomial, and let $\theta$ be the greatest element of $\supp(g)$. Also let $f\in\Omega$. Then, it holds that:

{\bf{(1)}}\,\,For any $\gamma\in\max(\supp(f))$, we have $(f\cdot g)_{[\gamma+\theta]}=f_{[\gamma]}$;

{\bf{(2)}}\,\,$\max(\supp(f\cdot g))=\max(\supp(f))+\{\theta\}$;

{\bf{(3)}}\,\,$\Delta(\supp(f\cdot g))=\Delta(\supp(f))+\Delta(\supp(g))$;

{\bf{(4)}}\,\,$\deg(f\cdot g)=\deg(f)+\deg(g)$. In particular, $f\neq0\Longrightarrow f\cdot g\neq0$;

{\bf{(5)}}\,\,Let $A\subseteq\mathbb{N}^{n}$, $h\in\Omega$ such that $A+\supp(h)\subseteq\Delta(\supp(f))$. Then, we have $A+\supp(h\cdot g)\subseteq\Delta(\supp(f\cdot g))$.
\end{lemma}

\setlength{\parindent}{2em}
The following well known lemma, which will be used in Section 4, can be verified via some straightforward computation (see, e.g., \cite{11,23,24}).

\setlength{\parindent}{0em}
\begin{lemma}
Let $f\in\Omega$. Then, for any $(u,\alpha)\in R^{n}\times\mathbb{N}^{n}$, it holds that
$$\mbox{$f(x_1+u_1,\dots,x_n+u_n)_{[\alpha]}=\sum_{\gamma\in\supp(f)}\Phi((u,\alpha),\gamma)\cdot f_{[\gamma]}$}.$$
\end{lemma}

\setlength{\parindent}{2em}
The following lemma, which will be used in Sections 4 and 5, is straightforward to verify.

\setlength{\parindent}{0em}
\begin{lemma}
Let $C\subseteq\mathbb{N}^{n}$. Then, $\mathbb{N}^{n}-\nabla(C)$ is finite if and only if for any $k\in[1,n]$, there exists $\beta\in C$ such that $(\forall~l\in[1,n]-\{k\}:\beta_l=0)$.
\end{lemma}

\setlength{\parindent}{2em}
The following three lemmas will be used in Section 5. Their proofs are included in Appendices B and C for completeness.

\setlength{\parindent}{0em}
\begin{lemma}
Let $(g_1,\dots,g_n)\in\Omega^{n}$ be a family of monic polynomials such that $g_k\in R[x_{k}]$ for all $k\in[1,n]$. Also fix $f\in\Omega$ such that $g_k\mid f$ for all $k\in[1,n]$. Then, it holds that $(\prod_{k=1}^{n}g_k)\mid f$.
\end{lemma}

\setlength{\parindent}{0em}
\begin{lemma}
Let $\alpha\in\mathbb{N}^{n}$, $t\in\mathbb{N}$, and let
$$\mbox{$B=\{(\alpha_1\theta_1,\dots,\alpha_n\theta_n)\mid \theta\in\mathbb{N}^{n},\sum_{i=1}^{n}\theta_i=t\}$}.$$
Then, it holds that
\begin{equation}\mbox{$|\mathbb{N}^{n}-\nabla(B)|=\left(\prod_{i=1}^{n}\alpha_i\right)\cdot\binom{n+t-1}{n}$}.\end{equation}
Assume in addition that $\alpha\in(\mathbb{Z}^{+})^{n}$. Then, it holds that
\begin{equation}\mbox{$\mathbb{N}^{n}-\nabla(B)=\{\beta\in\mathbb{N}^{n}\mid\sum_{i=1}^{n}\left\lfloor\frac{\beta_i}{\alpha_i}\right\rfloor\leqslant t-1\}$}.\end{equation}
\end{lemma}

\setlength{\parindent}{0em}
\begin{lemma}
Let $\alpha\in\mathbb{N}^{n}$, $\gamma\in\mathbb{N}^{n}$ with $\gamma\leqslant\alpha$, and fix $t\in\mathbb{Z}^{+}$. Let
$$\mbox{$B=\{(\alpha_1\theta_1,\dots,\alpha_n\theta_n)\mid \theta\in\mathbb{N}^{n},\sum_{i=1}^{n}\theta_i=t\}$},$$
$$\mbox{$C=\{(\alpha_1\theta_1,\dots,\alpha_n\theta_n)+\alpha-\gamma\mid \theta\in\mathbb{N}^{n},\sum_{i=1}^{n}\theta_i=t-1\}$}.$$
Then, it holds that
$$\mbox{$|\mathbb{N}^{n}-\nabla(B\cup C)|=\left(\prod_{i=1}^{n}\alpha_i\right)\cdot\binom{n+t-1}{n}-\left(\prod_{i=1}^{n}\gamma_i\right)\cdot\binom{n+t-2}{n-1}$}.$$
\end{lemma}

\setlength{\parindent}{2em}
We end this section by considering a class of Gr\"{o}bner bases which will be used in Section 5. The first part of the following lemma is straightforward to verify, and the second part will be established in Appendix D.

\setlength{\parindent}{0em}
\begin{lemma}
Let $(g_1,\dots,g_n)\in\Omega^{n}$ be a family of monic polynomials such that $g_k\in R[x_{k}]$ for all $k\in[1,n]$. Then, the following two statements hold:

{\bf{(1)}}\,\,For any $\alpha\in\mathbb{N}^{n}$, $\prod_{k=1}^{n}{g_{k}}^{\alpha_{k}}$ is monic, and $(\deg(g_1)\alpha_{1},\cdots,\deg(g_n)\alpha_{n})$ is the greatest element of $\supp(\prod_{k=1}^{n}{g_{k}}^{\alpha_{k}})$;

{\bf{(2)}}\,\,Let $A$ be a finite subset of $\mathbb{N}^{n}$. Then, $(\prod_{k=1}^{n}{g_k}^{\alpha_{k}}\mid\alpha\in A)$ is a Gr\"{o}bner basis of $\langle\{\prod_{k=1}^{n}{g_k}^{\alpha_{k}}\mid\alpha\in A\}\rangle$.
\end{lemma}

\section{Characterizations of Gr\"{o}bner bases}

\setlength{\parindent}{0em}
\begin{lemma}
Let $\Lambda$ be a finite set, and let $(g(\lambda)\mid\lambda\in\Lambda)\in\Omega^{\Lambda}$ be a family of monic polynomials. For any $\lambda\in\Lambda$, let $\theta(\lambda)$ be the greatest element of $\supp(g(\lambda))$. Then, for any $f\in\Omega$, there exist $(p(\lambda)\mid\lambda\in\Lambda)\in\Omega^{\Lambda}$ and $\tau\in\Omega$ satisfying the following four conditions:

{\bf{(1)}}\,\,$f=(\sum_{\lambda\in\Lambda}p(\lambda)\cdot g(\lambda))+\tau$;

{\bf{(2)}}\,\,For any $\lambda\in\Lambda$, it holds that $\supp(p(\lambda))+\supp(g(\lambda))\subseteq\Delta(\supp(f))$;

{\bf{(3)}}\,\,For any $\lambda\in\Lambda$ and $\alpha\in\supp(\tau)$, it holds that $\theta(\lambda)\nleqslant\alpha$;

{\bf{(4)}}\,\,$\supp(\tau)\subseteq\Delta(\supp(f))$.
\end{lemma}

\begin{proof}
Let $\Gamma\triangleq\{A\subseteq\mathbb{N}^{n}\mid\text{$A$ is finite}\}$, and let $\preccurlyeq$ be an order on $\Gamma$ defined as follows: For any $A,B\in\Gamma$, $A\preccurlyeq B$ if and only if either $A=B$ or there exists $\beta\in B-A$ such that $\alpha\leqslant \beta$ for all $\alpha\in A-B$. One can check that $(\Gamma,\preccurlyeq)$ satisfies the minimal condition, i.e., for any $H\subseteq\Gamma$ with $H\neq\emptyset$, $H$ contains a minimal element with respect to $\preccurlyeq$.

\hspace*{4mm}\,\,Now let $f\in\Omega$. If $\theta(\lambda)\nleqslant\alpha$ for all $\lambda\in\Lambda$ and $\alpha\in\supp(f)$, then $(p(\lambda)\mid\lambda\in\Lambda)$ defined as $p(\lambda)=0$ for all $\lambda\in\Lambda$ and $\tau=f$ satisfy Conditions (1)--(4). Therefore in the following, we assume that there exist $t\in\Lambda$, $\gamma\in\supp(f)$ such that $\theta(t)\leqslant\gamma$. Now let
\begin{equation}\mbox{$h\triangleq f-f_{[\gamma]}\cdot\left(\prod_{i=1}^{n}{x_{i}}^{\gamma_{i}-\theta(t)_{i}}\right)\cdot g(t)$}.\end{equation}
By (3.1) and the fact that $\theta(t)$ is the greatest element of $\supp(g(t))$, we have
\begin{equation}\hspace*{-1mm}\supp(h)-\supp(f)\subseteq\{\gamma-\theta(t)\}+\supp(g(t))\subseteq\Delta(\{\gamma\})\subseteq\Delta(\supp(f)).\end{equation}

From $g(t)_{[\theta(t)]}=1_{R}$ and (3.1), we deduce that $h_{[\gamma]}=f_{[\gamma]}-f_{[\gamma]}=0$, which implies that $\gamma\in\supp(f)-\supp(h)$. This, together with (3.2), implies that $\supp(h)\preccurlyeq\supp(f)$, $\supp(h)\neq\supp(f)$, $\supp(h)\subseteq\Delta(\supp(f))$, which further implies that $\Delta(\supp(h))\subseteq\Delta(\supp(f))$.

\hspace*{4mm}\,\,Since $(\Gamma,\preccurlyeq)$ satisfies the minimal condition, applying an induction argument to $h$, we can choose $(q(\lambda)\mid\lambda\in\Lambda)\in\Omega^{\Lambda}$ and $\tau\in\Omega$ satisfying the following three conditions:

$(i)$\,\,$h=(\sum_{\lambda\in\Lambda}q(\lambda)\cdot g(\lambda))+\tau$;

$(ii)$\,\,For any $\lambda\in\Lambda$, it holds that $\supp(q(\lambda))+\supp(g(\lambda))\subseteq\Delta(\supp(h))$;

$(iii)$\,\,For any $\lambda\in\Lambda$ and $\alpha\in\supp(\tau)$, it holds that $\theta(\lambda)\nleqslant\alpha$.

Now define $(p(\lambda)\mid\lambda\in\Lambda)\in\Omega^{\Lambda}$ as $p(\lambda)=q(\lambda)$ for all $\lambda\in\Lambda-\{t\}$, and $$\mbox{$p(t)=q(t)+f_{[\gamma]}\cdot(\prod_{i=1}^{n}{x_{i}}^{\gamma_{i}-\theta(t)_{i}})$}.$$
We claim that Conditions (1)--(4) hold true for $f$, $(p(\lambda)\mid\lambda\in\Lambda)$ and $\tau$. Indeed, via some straightforward verification, (1) follows from $(i)$ and (3.1), (2) follows from $(ii)$, (3.2) and the fact that $\Delta(\supp(h))\subseteq\Delta(\supp(f))$, (3) follows from $(iii)$, and (4) follows from (1) and (2), as desired.
\end{proof}

\setlength{\parindent}{2em}
Now we are ready to prove the main result of this section.

\setlength{\parindent}{0em}
\begin{theorem}
Let $\Lambda$ be a finite set, and let $(g(\lambda)\mid\lambda\in\Lambda)\in\Omega^{\Lambda}$ be a family of monic polynomials. For any $\lambda\in\Lambda$, let $\theta(\lambda)$ be the greatest element of $\supp(g(\lambda))$. Moreover, let $Q$ be an ideal of $\Omega$ with $\{g(\lambda)\mid\lambda\in\Lambda\}\subseteq Q$. Then, the following six statements are equivalent to each other:

{\bf{(1)}}\,\,$(g(\lambda)\mid\lambda\in\Lambda)$ is a Gr\"{o}bner basis of $Q$;

{\bf{(2)}}\,\,$Q\cap\delta(\mathbb{N}^{n}-\nabla(\{\theta(\lambda)\mid\lambda\in\Lambda\}))=\{0\}$;

{\bf{(3)}}\,\,For any $f\in Q$, there exists $(p(\lambda)\mid\lambda\in\Lambda)\in\Omega^{\Lambda}$ satisfying the following two conditions:

\hspace*{6mm}\,\,{\bf{3.1)}}\,\,$f=\sum_{\lambda\in\Lambda}p(\lambda)\cdot g(\lambda)$;

\hspace*{6mm}\,\,{\bf{3.2)}}\,\,For any $\lambda\in\Lambda$, it holds that $\supp(p(\lambda))+\supp(g(\lambda))\subseteq\Delta(\supp(f))$;

{\bf{(4)}}\,\,For any $f\in Q$ and $\beta\in\max(\supp(f))$, there exists $\xi\in\Lambda$ with $\theta(\xi)\leqslant\beta$. Alternatively speaking, for any $h\in \Omega$ such that there exists $\gamma\in\max(\supp(h))$ with $(\forall~\lambda\in\Lambda:\theta(\lambda)\nleqslant\gamma)$, it holds that $h\not\in Q$;

{\bf{(5)}}\,\,For any $f\in Q$, there exists $(p(\lambda)\mid\lambda\in\Lambda)\in\Omega^{\Lambda}$ satisfying the following two conditions:

\hspace*{6mm}\,\,{\bf{5.1)}}\,\,$f=\sum_{\lambda\in\Lambda}p(\lambda)\cdot g(\lambda)$;

\hspace*{6mm}\,\,{\bf{5.2)}}\,\,For any $\lambda\in\Lambda$, it holds that $\deg(p(\lambda))+\deg(g(\lambda))\leqslant\deg(f)$;

{\bf{(6)}}\,\,For any $f\in Q-\{0\}$ and $\beta\in\supp(f)$ with $\deg(f)=\sum_{i=1}^{n}\beta_i$, there exists $\xi\in\Lambda$ with $\theta(\xi)\leqslant\beta$. Alternatively speaking, for any $h\in \Omega-\{0\}$ such that there exists $\gamma\in\supp(h)$ with $\deg(h)=\sum_{i=1}^{n}\gamma_i$ and $(\forall~\lambda\in\Lambda:\theta(\lambda)\nleqslant\gamma)$, it holds that $h\not\in Q$.
\end{theorem}

\begin{proof}
We begin by noting that $(1)\Longleftrightarrow(2)$ immediately follows from Definition 2.3, (2.5) and (2.10). Therefore in what follows, we will show that $(1)\Longrightarrow(3)\Longrightarrow(4)\Longrightarrow(6)\Longrightarrow(1)$ and $(3)\Longrightarrow(5)\Longrightarrow(6)$.

\hspace*{4mm}\,\,$(1)\Longrightarrow(3)$\,\,Let $f\in Q$. By Lemma 3.1, we can choose $(p(\lambda)\mid\lambda\in\Lambda)\in\Omega^{\Lambda}$ and $\tau\in\Omega$ satisfying the following three conditions:

$(i)$\,\,$f=(\sum_{\lambda\in\Lambda}p(\lambda)\cdot g(\lambda))+\tau$;

$(ii)$\,\,For any $\lambda\in\Lambda$, it holds that $\supp(p(\lambda))+\supp(g(\lambda))\subseteq\Delta(\supp(f))$;

$(iii)$\,\,For any $\lambda\in\Lambda$ and $\alpha\in\supp(\tau)$, it holds that $\theta(\lambda)\nleqslant\alpha$.

By $(i)$, we have $\tau\in Q$, which, together with $(iii)$ and Definition 2.3, further implies that $\tau=0$. Now (3) immediately follows from $(i)$ and $(ii)$.

\hspace*{4mm}\,\,$(3)\Longrightarrow(4)$\,\,Let $f\in Q$. By (3), we can choose $(p(\lambda)\mid\lambda\in\Lambda)\in\Omega^{\Lambda}$ satisfying 3.1) and 3.2). Consider an arbitrary $\beta\in\max(\supp(f))$. By 3.1), there exists $\xi\in\Lambda$ such that $\beta\in\supp(p(\xi)\cdot g(\xi))$. Hence we can choose $\alpha\in\supp(p(\xi))$, $\gamma\in\supp(g(\xi))$ with $\beta=\alpha+\gamma$. Since $\theta(\xi)$ is the greatest element of $\supp(g(\xi))$, we have $\gamma\leqslant\theta(\xi)$, which further yields that $\beta\leqslant\alpha+\theta(\xi)$. By 3.2), we have $\alpha+\theta(\xi)\in\Delta(\supp(f))$. Thus we can choose $\mu\in\supp(f)$ with $\alpha+\theta(\xi)\leqslant\mu$. This, together with $\beta\leqslant\alpha+\theta(\xi)$, implies that $\beta\leqslant\mu$. From $\beta\in\max(\supp(f))$, $\mu\in\supp(f)$, we deduce that $\beta=\mu$. It then follows that $\theta(\xi)\leqslant\alpha+\theta(\xi)=\beta$, as desired.

\hspace*{4mm}\,\,$(4)\Longrightarrow(6)$\,\,Let $f\in Q-\{0\}$ and $\beta\in\supp(f)$ with $\deg(f)=\sum_{i=1}^{n}\beta_i$. Then, we have $\beta\in\max(\supp(f))$, and the desired result follows from (4).

\hspace*{4mm}\,\,$(6)\Longrightarrow(1)$\,\,Consider an arbitrary $f\in Q-\{0\}$. Since $f\neq0$, we can choose $\beta\in\supp(f)$ with $\deg(f)=\sum_{i=1}^{n}\beta_i$. By (6), we can further choose $\xi\in\Lambda$ with $\theta(\xi)\leqslant\beta$. Now (1) immediately follows from Definition 2.3.

\hspace*{4mm}\,\,$(3)\Longrightarrow(5)$\,\,This follows from the fact that 3.2) implies 5.2).

\hspace*{4mm}\,\,$(5)\Longrightarrow(6)$\,\,Let $f\in Q-\{0\}$. By (5), we can choose $(p(\lambda)\mid\lambda\in\Lambda)\in\Omega^{\Lambda}$ satisfying 5.1) and 5.2). Consider an arbitrary $\beta\in\supp(f)$ with $\deg(f)=\sum_{i=1}^{n}\beta_i$. By 5.1), there exists $\xi\in\Lambda$ with $\beta\in\supp(p(\xi)\cdot g(\xi))$. By 5.2), we have $\deg(p(\xi)\cdot g(\xi))\leqslant\sum_{i=1}^{n}\beta_i$. It then follows that $\beta\in\max(\supp(p(\xi)\cdot g(\xi)))$, which, together with (2) of Lemma 2.1, implies that $\beta\in\max(\supp(p(\xi)))+\{\theta(\xi)\}$, and hence $\theta(\xi)\leqslant\beta$, as desired.
\end{proof}

\begin{remark}
In the study of combinatorial Nullstellensatz, conclusions of the form similar to (5), (6) of Theorem 3.1 have been established in various settings: in Alon's combinatorial Nullstellensatz Theorems 1.1 and 1.2, in the Nullstellensatz with multiplicity [9, Theorem 3.1, Corollary 3.2], in the Nullstellensatz for multisets [24, Theorems 1 and 6], in Mezei's Nullstellensatz [29, Theorem 6.12], and in Nullstellensatz over a commutative ring, including [34, Theorem 7.3], [30, Theorem 1], [17, Theorem 3.8 and c) of Theorem 3.9], [25, Theorems 2.1 and 2.2] and [23, Theorems 3 and 4]. Moreover, conclusions of the form similar to (4) of Theorem 3.1 have been established in [27, Theorem 2], which is a generalization of Theorem 1.2, and in [11, Corollary 1.6], which considers Nullstellensatz for multisets. In this sense, Theorem 3.1 reveals the essence of Alon's combinatorial Nullstellensatz and some of its generalizations from a Gr\"{o}bner basis perspective. We also note that to the best of our knowledge, (3) of Theorem 3.1 seems to be a new characterization.
\end{remark}

\setlength{\parindent}{2em}
As a first application of Theorem 3.1, in the following example, we recover Alon's combinatorial Nullstellensatz (Theorems 1.1 and 1.2) and Laso\'{n}'s generalization of Theorem 1.2 (see [27, Theorem 2]), and also derive an equivalent version of Theorem 1.1.

\setlength{\parindent}{0em}
\begin{example}
Suppose that $R$ is an integral domain. Let $S_1,\dots,S_n$ be nonempty finite subsets of $R$, $g_i=\prod_{u\in S_i}(x_i-u)$ for all $i\in[1,n]$, and let
$$\mbox{$Q=\{f\in\Omega\mid \text{$f(a_1,\dots,a_n)=0$ for all $a\in\prod_{i=1}^{n}S_i$}\}$}.$$
Apparently, $(g_1,\dots,g_n)$ is a tuple of monic polynomials; for any $k\in[1,n]$, $\theta(k)\triangleq(\underbrace{0,\dots,0}_{k-1\mbox{\scriptsize}},|S_k|,\underbrace{0,\dots,0}_{n-k\mbox{\scriptsize}})$ is the greatest element of $\supp(g_k)$; and moreover, $Q$ is an ideal of $\Omega$ with $\{g_1,\dots,g_n\}\subseteq Q$. By [3, Lemma 2.1], we have $$Q\cap\delta(\mathbb{N}^{n}-\nabla(\{\theta(1),\dots,\theta(n)\}))=\{0\}.$$
It then follows from Theorem 3.1 that $(g_1,\dots,g_n)$ is a Gr\"{o}bner basis of $Q$, as has been shown in [32, Section 6]. Moreover, (5) of Theorem 3.1 boils down to Theorem 1.1, (6) of Theorem 3.1 boils down to Theorem 1.2, (4) of Theorem 3.1 boils down to Laso\'{n}'s result [27, Theorem 2], and (3) of Theorem 3.1 boils down to an equivalent version of Theorem 1.1.
\end{example}

\section{A Nullstellensatz derived from Theorem 3.1}

\setlength{\parindent}{2em}
Throughout this section, we let $S_{1},\dots,S_{n}$ be finite subsets of $R$, and let $(B_a\mid a\in\prod_{i=1}^{n}S_{i})$ be a tuple of subsets of $\mathbb{N}^{n}$ such that for any $a\in\prod_{i=1}^{n}S_{i}$, $\mathbb{N}^{n}-\nabla(B_a)$ is finite. Moreover, we let
\begin{equation}\mbox{$\zeta_1\triangleq\sum_{a\in\prod_{i=1}^{n}S_{i}}|\mathbb{N}^{n}-\nabla(B_{a})|$},\end{equation}
\begin{equation}\hspace*{-11mm}Q\triangleq\{f\in\Omega\mid \mbox{$\supp(f(x_{1}+a_{1},\dots,x_{n}+a_{n}))\subseteq\nabla(B_{a})$ for all $a\in\prod_{i=1}^{n}S_{i}$}\}.\end{equation}
It is straightforward to verify that $Q$ is an ideal of $\Omega$.

\setlength{\parindent}{0em}
\begin{lemma}
Suppose that for any $i\in[1,n]$, $S_{i}$ satisfies Condition (F) in $R$. Then, $\Omega/Q$ is a finitely generated free $R$-module with $\rank_{R}(\Omega/Q)=\zeta_1$.
\end{lemma}

\begin{proof}
First, consider an arbitrary $a\in\prod_{i=1}^{n}S_{i}$. Let
$$J_a\triangleq\{f\in\Omega\mid \supp(f(x_{1}+a_{1},\dots,x_{n}+a_{n}))\subseteq\nabla(B_{a})\}.$$
Apparently, $J_a$ is an ideal of $\Omega$. Moreover, since $\mathbb{N}^{n}-\nabla(B_a)$ is finite, it follows from Lemma 2.3 that for any $m\in[1,n]$, there exists $p\in\mathbb{N}$ with $(x_m-a_m)^{p}\in J_a$.

\hspace*{4mm}\,\,Next, let $b,d\in\prod_{i=1}^{n}S_{i}$ with $b\neq d$. We will show that
\begin{equation}J_b+J_d=\Omega.\end{equation}
Indeed, fix $m\in[1,n]$ with $b_{m}\neq d_{m}$. Since $S_{m}$ satisfies Condition (F) in $R$, by Definition 2.2, $d_m-b_m$ is a multiplicative unit of $R$. It then follows that $\langle\{x_{m}-b_{m},x_{m}-d_{m}\}\rangle=\Omega$, which implies that $\langle\{(x_{m}-b_{m})^{p},(x_{m}-d_{m})^{q}\}\rangle=\Omega$ for all $p,q\in\mathbb{N}$. By the previous paragraph, we can choose $p,q\in\mathbb{N}$ such that $(x_m-b_m)^{p}\in J_b$, $(x_m-d_m)^{q}\in J_d$, which further implies (4.3), as desired.

\hspace*{4mm}\,\,Note that if $\prod_{i=1}^{n}S_{i}=\emptyset$, then the desired result can be readily verified. Therefore in what follows, we assume $\prod_{i=1}^{n}S_{i}\neq\emptyset$. Then, we have $Q=\bigcap_{a\in\prod_{i=1}^{n}S_{i}}J_a$. For any $a\in\prod_{i=1}^{n}S_{i}$, it is straightforward to verify that $\Omega/J_a$ is a finitely generated free $R$-module with $\rank_{R}(\Omega/J_a)=|\mathbb{N}^{n}-\nabla(B_{a})|$. This, together with (4.3) and the Chinese remainder theorem (see [6, Proposition 1.10]), leads to the fact that $\Omega/Q$ is a finitely generated free $R$-module with
$$\mbox{$\rank_{R}(\Omega/Q)=\sum_{a\in\prod_{i=1}^{n}S_{i}}\rank_{R}(\Omega/J_a)=\sum_{a\in\prod_{i=1}^{n}S_{i}}|\mathbb{N}^{n}-\nabla(B_{a})|=\zeta_1$},$$
as desired.
\end{proof}

\setlength{\parindent}{2em}
From now on, we let $\Lambda$ be a finite set, and let $(g(\lambda)\mid\lambda\in\Lambda)\in Q^{\Lambda}$ be a family of monic polynomials. For any $\lambda\in\Lambda$, let $\theta(\lambda)$ be the greatest element of $\supp(g(\lambda))$. Moreover, let $D=\{\theta(\lambda)\mid\lambda\in\Lambda\}$. Throughout the remainder of this section, we assume that $\mathbb{N}^{n}-\nabla(D)$ is finite, and let
\begin{equation}\mbox{$\zeta_2\triangleq|\mathbb{N}^{n}-\nabla(D)|$}.\end{equation}

\setlength{\parindent}{0em}
\begin{lemma}
{\bf{(1)}}\,\,Suppose that for any $i\in[1,n]$, $S_{i}$ satisfies Condition (F) in $R$. Then, we have $\zeta_1\leqslant \zeta_2$.
Moreover, if $\zeta_1=\zeta_2$, then $(g(\lambda)\mid\lambda\in\Lambda)$ is a Gr\"{o}bner basis of $Q$.

{\bf{(2)}}\,\,Write $G=\bigcup_{a\in\prod_{i=1}^{n}S_{i}}\{a\}\times(\mathbb{N}^{n}-\nabla(B_{a}))$, and suppose that $(g(\lambda)\mid\lambda\in\Lambda)$ is a Gr\"{o}bner basis of $Q$. Then, we have $\zeta_2\leqslant \zeta_1$. Assume in addition that $\zeta_1=\zeta_2$. Then, for any $\sigma\in R^{G}$, it holds that
\begin{equation}(\mbox{$\sum_{(v,\alpha)\in G}\sigma_{(v,\alpha)}\cdot \Phi((v,\alpha),\gamma)=0$ for all $\gamma\in \mathbb{N}^{n}-\nabla(D)$})\Longrightarrow\sigma=0.\end{equation}
\end{lemma}

\begin{proof}
{\bf{(1)}}\,\,We note that $\delta(\mathbb{N}^{n}-\nabla(D))$ is a finitely generated free $R$-module with $\rank_{R}(\delta(\mathbb{N}^{n}-\nabla(D)))=\zeta_2$. Define the $R$-module homomorphism $\pi:\delta(\mathbb{N}^{n}-\nabla(D))\longrightarrow\Omega/Q$ as
$\pi(h)=h+Q$. It follows from Lemma 3.1 that $\pi$ is surjective. This, together with Lemma 4.1, implies that $\zeta_1=\rank_{R}(\Omega/Q)\leqslant \zeta_2$, as desired. Now assume in addition that $\zeta_1=\zeta_2$. Then, by Lemma 4.1, $\Omega/Q$ and $\delta(\mathbb{N}^{n}-\nabla(D))$ are finitely generated free $R$-modules with the same rank. This, together with the surjectivity of $\pi$ and [14, Theorem 5.36], implies that $\pi$ is injective, and hence $Q\cap\delta(\mathbb{N}^{n}-\nabla(D))=\{0\}$. Now an application of Theorem 3.1 yields that $(g(\lambda)\mid\lambda\in\Lambda)$ is a Gr\"{o}bner basis of $Q$, as desired.

{\bf{(2)}}\,\,By Theorem 3.1, we have $Q\cap\delta(\mathbb{N}^{n}-\nabla(D))=\{0\}$. Let $W\triangleq\Phi\mid_{G\times (\mathbb{N}^{n}-\nabla(D))}$. Then, $W$ can be regarded as a matrix over $R$ whose rows are indexed by $G$ and columns are indexed by $\mathbb{N}^{n}-\nabla(D)$. Consider the following homogeneous system of linear equations with unknowns $\rho\in R^{\mathbb{N}^{n}-\nabla(D)}$:
\begin{equation}\mbox{$\sum_{\gamma\in \mathbb{N}^{n}-\nabla(D)}W((v,\alpha),\gamma)\cdot \rho_{\gamma}=0$ for all $(v,\alpha)\in G$}.\end{equation}
Let $\rho\in R^{\mathbb{N}^{n}-\nabla(D)}$ be a solution to (4.6), and we will show that $\rho=0$. Indeed, let $f=\sum_{\gamma\in \mathbb{N}^{n}-\nabla(D)}\rho_{\gamma}\cdot(\prod_{k=1}^{n}{x_k}^{\gamma_k})$. For any $(v,\alpha)\in G$, by Lemma 2.2 and $\supp(f)\subseteq \mathbb{N}^{n}-\nabla(D)$, we have
$$\mbox{$f(x_1+v_1,\dots,x_n+v_n)_{[\alpha]}=\sum_{\gamma\in \mathbb{N}^{n}-\nabla(D)}W((v,\alpha),\gamma)\cdot \rho_{\gamma}=0$}.$$
It follows that $f\in Q\cap\delta(\mathbb{N}^{n}-\nabla(D))=\{0\}$, which further implies that $\rho=0$, as desired. Therefore (4.6) only has trivial solution. Hence an application of McCoy's theorem [14, Theorem 5.3] yields that $|\mathbb{N}^{n}-\nabla(D)|\leqslant|G|$, and hence $\zeta_2\leqslant \zeta_1$, as desired. Now, assume in addition that $\zeta_1=\zeta_2$. Then, we have $|G|=|\mathbb{N}^{n}-\nabla(D)|$, and (4.6) only has trivial solution. It again follows from [14, Theorem 5.3] that (4.5) holds true, as desired.
\end{proof}

\setlength{\parindent}{2em}
Now we present some necessary and sufficient conditions for $(g(\lambda)\mid\lambda\in\Lambda)$ to be a Gr\"{o}bner basis of $Q$.

\setlength{\parindent}{0em}
\begin{theorem}
{\bf{(1)}}\,\,Suppose that for any $i\in[1,n]$, $S_{i}$ satisfies Condition (D) in $R$. Then, it holds that $\zeta_1\leqslant \zeta_2$. Moreover, $(g(\lambda)\mid\lambda\in\Lambda)$ is a Gr\"{o}bner basis of $Q$ if and only if $\zeta_1=\zeta_2$.

{\bf{(2)}}\,\,Suppose that $\prod_{i=1}^{n}S_{i}\neq\emptyset$, $\mathbf{0}\not\in B_a$ for all $a\in\prod_{i=1}^{n}S_{i}$, $\zeta_1\leqslant \zeta_2$, and $(g(\lambda)\mid\lambda\in\Lambda)$ is a Gr\"{o}bner basis of $Q$. Then, it holds that $\zeta_1=\zeta_2$, and for any $k\in[1,n]$, $S_{k}$ satisfies Condition (D) in $R$.
\end{theorem}

\begin{proof}
{\bf{(1)}}\,\,Let $E=\{c\in R\mid \text{$c$ is not a zero divisor of $R$}\}$, and let $T$ denote the ring of fractions of $R$ with respect to $E$ (see [6, Chapter 3]). Moreover, let $\Omega_1\triangleq T[x_1,\dots,x_n]$
denote the polynomial ring over $T$ in $n$ variables $x_1,\dots,x_n$. For any $h\in\Omega_1$, we let $\supp_1(h)=\{\alpha\in\mathbb{N}^{n}\mid h_{[\alpha]}\neq0\}$, where for any $\alpha\in\mathbb{N}^{n}$, $h_{[\alpha]}$ denotes the coefficient of $\prod_{i=1}^{n}{x_i}^{\alpha_i}$ in $h$.

\hspace*{4mm}\,\,The following two facts are straightforward to verify.

{\bf{Fact 1.}}\,\,$\varphi:R\longrightarrow T$ defined as $\varphi(a)=a/1_{R}$ is injective.

{\bf{Fact 2.}}\,\,Let $X\subseteq R$ such that $X$ satisfies Condition (D) in $R$, and let $Y=\{b/1_{R}\mid b\in X\}$. Then, for any $y_1,y_2\in Y$ with $y_1\neq y_2$, $y_2-y_1$ is a multiplicative unit of $T$, that is to say, $Y$ satisfies Condition (F) in $T$.

Based on Fact 1, we define $(P_1,\dots,P_n)$, $(L_e\mid e\in\prod_{i=1}^{n}P_i)$ and $Q_1$ as follows:
$$\mbox{$\forall~k\in[1,n]:P_k=\{u/1_{R}\mid u\in S_k\}$},$$
\begin{equation}\mbox{$\forall~a\in\prod_{i=1}^{n}S_i:L_{(a_1/1_{R},\dots,a_n/1_{R})}=B_a$};\end{equation}
$$Q_1\triangleq\{\eta\in\Omega_1\mid \mbox{$\supp_1(\eta(x_{1}+e_{1},\dots,x_{n}+e_{n}))\subseteq\nabla(L_{e})$ for all $e\in\prod_{i=1}^{n}P_{i}$}\}.$$
By Fact 2, for any $k\in[1,n]$, $P_{k}$ is a finite subset of $T$, and $P_{k}$ satisfies Condition (F) in $T$. By (4.7), for any $e\in\prod_{i=1}^{n}P_{i}$, it holds that $L_e\subseteq\mathbb{N}^{n}$, $\mathbb{N}^{n}-\nabla(L_e)$ is finite. Moreover, it holds that
\begin{equation}\mbox{$\sum_{e\in\prod_{i=1}^{n}P_{i}}|\mathbb{N}^{n}-\nabla(L_{e})|=\sum_{a\in\prod_{i=1}^{n}S_{i}}|\mathbb{N}^{n}-\nabla(B_a)|=\zeta_1$}.\end{equation}
We also note that $Q_1$ is an ideal of $\Omega_1$. Now define $(h(\lambda)\mid\lambda\in\Lambda)\in{\Omega_1}^{\Lambda}$ as
$$\text{$\forall~\lambda\in\Lambda:h(\lambda)_{[\alpha]}=g(\lambda)_{[\alpha]}/1_{R}$ for all $\alpha\in\mathbb{N}^{n}$}.$$
Consider an arbitrary $\mu\in\Lambda$. From Fact 1, we deduce that $\supp_1(h(\mu))=\supp(g(\mu))$, $\theta(\mu)$ is the greatest element of $\supp_1(h(\mu))$, $h(\mu)_{[\theta(\mu)]}=1_T$. Moreover, $g(\mu)\in Q$ implies that $h(\mu)\in Q_1$. Now with (4.8), an application of (1) of Lemma 4.2 to $T$ and $\Omega_1$ implies that $\zeta_1\leqslant \zeta_2$, as desired.

\hspace*{4mm}\,\,Next, we prove the ``moreover'' assertion. Note that the ``only if'' part follows from (2) of Lemma 4.2 and the proven fact that $\zeta_1\leqslant \zeta_2$, it remains to prove the ``if'' part. To this end, suppose that $\zeta_1=\zeta_2$. Then, with (4.8), an application of (1) of Lemma 4.2 to $T$ and $\Omega_1$ implies that in $\Omega_1$, $(h(\lambda)\mid\lambda\in\Lambda)$ is a Gr\"{o}bner basis of $Q_1$. Now consider an arbitrary $\tau\in Q-\{0\}$. Let $\eta\in\Omega_1$ be defined as $\eta_{[\alpha]}=\tau_{[\alpha]}/1_{R}$ for all $\alpha\in\mathbb{N}^{n}$. By Fact 1, we have $\supp(\tau)=\supp_1(\eta)$. In particular, $\eta\neq0$. Moreover, $\tau\in Q$ implies that $\eta\in Q_1$. Since $(h(\lambda)\mid\lambda\in\Lambda)$ is a Gr\"{o}bner basis of $Q_1$, by Definition 2.3, we can choose $\alpha\in\supp_1(\eta)$ and $r\in\Lambda$ with $\theta(r)\leqslant\alpha$. Hence we have $\alpha\in\supp(\tau)$, $r\in\Lambda$, $\theta(r)\leqslant\alpha$. It then follows from Definition 2.3 that $(g(\lambda)\mid\lambda\in\Lambda)$ is a Gr\"{o}bner basis of $Q$, as desired.

{\bf{(2)}}\,\,Let $G=\bigcup_{a\in\prod_{i=1}^{n}S_{i}}\{a\}\times(\mathbb{N}^{n}-\nabla(B_{a}))$. First, it follows from (2) of Lemma 4.2 that $\zeta_1=\zeta_2$, as desired. Next, for any $a\in\prod_{i=1}^{n}S_{i}$, by $\mathbf{0}\not\in B_a$, we have $(a,\mathbf{0})\in G$. Now let $k\in[1,n]$, and we will show that $S_{k}$ satisfies Condition (D) in $R$. Indeed, let $c,d\in S_k$ with $c\neq d$. Since $\prod_{i=1}^{n}S_{i}\neq\emptyset$, we can choose $y,w\in\prod_{i=1}^{n}S_{i}$ such that $y_k=c$, $w_k=d$, and $y_j=w_j$ for all $j\in[1,n]-\{k\}$. Consider an arbitrary $b\in R$ with $b\cdot(d-c)=0$. Then, we can choose $\sigma\in R^{G}$ such that $\sigma_{(y,\mathbf{0})}=-b$, $\sigma_{(w,\mathbf{0})}=b$, and $\sigma_{(a,\beta)}=0$ for all $(a,\beta)\in G-\{(y,\mathbf{0}),(w,\mathbf{0})\}$. Let $\gamma\in \mathbb{N}^{n}$ be arbitrary. By $b\cdot(d-c)=0$, we have $b\cdot(d^{^{\gamma_k}}-c^{^{\gamma_k}})=0$, which, together with (2.11) and the definitions of $\sigma,y,w$, further implies that
\begin{eqnarray*}
\begin{split}
\mbox{$\sum_{(v,\alpha)\in G}\sigma_{(v,\alpha)}\cdot \Phi((v,\alpha),\gamma)$}&=\mbox{$b\cdot(\Phi((w,\mathbf{0}),\gamma)-\Phi((y,\mathbf{0}),\gamma))$}\\
&=\mbox{$b\cdot(d^{^{\gamma_k}}-c^{^{\gamma_k}})\cdot(\prod_{i\in[1,n]-\{k\}}{w_i}^{\gamma_i})$}\\
&=0.
\end{split}
\end{eqnarray*}
It then follows from (2) of Lemma 4.2 that $\sigma=0$, which further implies that $b=\sigma_{(w,\mathbf{0})}=0$. Therefore $d-c$ is not a zero divisor of $R$. Hence from Definition 2.2, we deduce that $S_{k}$ satisfies Condition (D) in $R$, as desired.
\end{proof}

\setlength{\parindent}{2em}
The following Nullstellensatz, which immediately follows from Theorems 3.1 and 4.1, is one of the main results in this paper.

\setlength{\parindent}{0em}
\begin{theorem}
Suppose that $\zeta_1=\zeta_2$, and for any $i\in[1,n]$, $S_{i}$ satisfies Condition (D) in $R$. Consider an arbitrary $f\in Q$. Then, there exists $(p(\lambda)\mid\lambda\in\Lambda)\in\Omega^{\Lambda}$ with $f=\sum_{\lambda\in\Lambda}p(\lambda)\cdot g(\lambda)$ and
$$\forall~\lambda\in\Lambda:\supp(p(\lambda))+\supp(g(\lambda))\subseteq\Delta(\supp(f)).$$
Moreover, for any $\beta\in\max(\supp(f))$, there exists $\lambda\in\Lambda$ with $\theta(\lambda)\leqslant\beta$.
\end{theorem}

\begin{remark}
Theorem 4.1 implies that with some additional assumptions, Condition (D) is a necessary and sufficient condition for $(g(\lambda)\mid\lambda\in\Lambda)$ to be a Gr\"{o}bner basis of $Q$. In fact, in Theorem 4.1, the assumption that all the $g(\lambda)$'s are monic can be weakened to that all the $g(\lambda)$'s have leading coefficient $1_R$ with respect to some fixed monomial order. Therefore Theorem 4.1 holds true in a more general setting. By contrast, our Nullstellensatz Theorem 4.2 relies heavily on the fact that all the $g(\lambda)$'s are monic.
\end{remark}

\section{Some corollaries of Theorems 4.1 and 4.2}
\setlength{\parindent}{2em}
In this section, we will use Theorems 4.1 and 4.2 to derive several Nullstellensatz.
\subsection{An extension of Mezei's Nullstellensatz}
\setlength{\parindent}{2em}
In [29, Theorem 6.12], Mezei establishes a general Nullstellensatz which includes both the Nullstellensatz for multisets in \cite{24} and the Nullstellensatz with multiplicity in \cite{9} as special cases. In this subsection, we further extend Mezei's result. Our extension will be an immediate corollary of Theorems 4.1 and 4.2.

Let $S_{1},\dots,S_{n}$ be finite subsets of $R$, $\Lambda$ be a finite set, and let
$$(\varepsilon(i,u,\lambda)\mid i\in[1,n],u\in S_{i},\lambda\in\Lambda)$$
be a tuple of natural numbers. Define $(g(\lambda)\mid \lambda\in\Lambda)$, $(\theta(\lambda)\mid\lambda\in\Lambda)$, $D$, $(B_a\mid a\in\prod_{i=1}^{n}S_{i})$ and $Q$ as follows:
\begin{equation}\mbox{$\forall~\lambda\in\Lambda:g(\lambda)=\prod_{i=1}^{n}\prod_{u\in S_{i}}(x_{i}-u)^{^{\varepsilon(i,u,\lambda)}}\in\Omega$},\end{equation}
\begin{equation}\mbox{$\forall~\lambda\in\Lambda:\theta(\lambda)=(\sum_{u\in S_{1}}\varepsilon(1,u,\lambda),\dots,\sum_{u\in S_{n}}\varepsilon(n,u,\lambda))\in\mathbb{N}^{n}$},\end{equation}
\begin{equation}D\triangleq\{\theta(\lambda)\mid\lambda\in\Lambda\}\subseteq\mathbb{N}^{n},\end{equation}
\begin{equation}\mbox{$\forall~a\in\prod_{i=1}^{n}S_{i}:B_{a}\triangleq\{(\varepsilon(1,a_{1},\lambda),\dots,\varepsilon(n,a_{n},\lambda))\mid\lambda\in\Lambda\}\subseteq\mathbb{N}^{n}$},\end{equation}
\begin{equation}\hspace*{-11mm}Q\triangleq\{f\in\Omega\mid \mbox{$\supp(f(x_{1}+a_{1},\dots,x_{n}+a_{n}))\subseteq\nabla(B_{a})$ for all $a\in\prod_{i=1}^{n}S_{i}$}\}.\end{equation}
We assume in addition that $\mathbb{N}^{n}-\nabla(D)$ is finite, and let
$$\zeta_2\triangleq|\mathbb{N}^{n}-\nabla(D)|.$$
By Lemma 2.3, one can check that for any $a\in\prod_{i=1}^{n}S_{i}$, $\mathbb{N}^{n}-\nabla(B_{a})$ is finite. As in Section 4, we let $$\mbox{$\zeta_1\triangleq\sum_{a\in\prod_{i=1}^{n}S_{i}}|\mathbb{N}^{n}-\nabla(B_{a})|$}.$$
Via some straightforward verification, we deduce that $Q$ is an ideal of $\Omega$, $(g(\lambda)\mid\lambda\in\Lambda)$ is a family of monic polynomials with $\{g(\lambda)\mid\lambda\in\Lambda\}\subseteq Q$; moreover, for any $\lambda\in\Lambda$, $\theta(\lambda)$ is the greatest element of $\supp(g(\lambda))$.

\setlength{\parindent}{0em}
\begin{lemma}
It holds that $\zeta_1\leqslant \zeta_2$.
\end{lemma}

\begin{proof}
First, for the special case that all the $S_{i}$'s satisfy Condition (F) in $R$, the desired result is an immediate corollary of (1) of Lemma 4.2.

\hspace*{4mm}\,\,Next, we consider the general case. To this end, let $P_{1},\dots,P_{n}$ be finite subsets of $\mathbb{C}$ with $|P_i|=|S_i|$ for all $i\in[1,n]$. For any $i\in[1,n]$, fix a bijective map $\omega_i:P_i\longrightarrow S_i$. Define a tuple $(\sigma(i,v,\lambda)\mid i\in[1,n],v\in P_{i},\lambda\in\Lambda)$ as $\sigma(i,v,\lambda)=\varepsilon(i,\omega_i(v),\lambda)$. It can be readily verified that
$$\mbox{$\forall~\lambda\in\Lambda:(\sum_{v\in P_{1}}\sigma(1,v,\lambda),\dots,\sum_{v\in P_{n}}\sigma(n,v,\lambda))=\theta(\lambda)$},$$
$$\mbox{$\forall~d\in\prod_{i=1}^{n}P_{i}:\{(\sigma(1,d_{1},\lambda),\dots,\sigma(n,d_{n},\lambda))\mid\lambda\in\Lambda\}=B_{(\omega_1(d_1),\dots,\omega_n(d_n))}$}.$$
Since $\mathbb{C}$ is a field, for any $i\in[1,n]$, $P_{i}$ satisfies Condition (F) in $\mathbb{C}$. It then follows from the previous paragraph that $$\mbox{$\zeta_1=\sum_{d\in\prod_{i=1}^{n}P_{i}}|\mathbb{N}^{n}-\nabla(B_{(\omega_1(d_1),\dots,\omega_n(d_n))})|\leqslant|\mathbb{N}^{n}-\nabla(D)|=\zeta_2$},$$
as desired.
\end{proof}

\setlength{\parindent}{2em}
Now we are ready to prove the main result of this subsection.

\setlength{\parindent}{0em}
\begin{theorem}
{\bf{(1)}}\,\,Suppose that $\zeta_1=\zeta_2$, and for any $k\in[1,n]$, $S_{k}$ satisfies Condition (D) in $R$. Then, $(g(\lambda)\mid\lambda\in\Lambda)$ is a Gr\"{o}bner basis of $Q$. Further consider an arbitrary $f\in Q$. Then, there exists $(p(\lambda)\mid\lambda\in\Lambda)\in\Omega^{\Lambda}$ such that $f=\sum_{\lambda\in\Lambda}p(\lambda)\cdot g(\lambda)$ and
\begin{equation}\forall~\lambda\in\Lambda:\supp(p(\lambda))+\supp(g(\lambda))\subseteq\Delta(\supp(f)).\end{equation}
Moreover, for an arbitrary $\beta\in\max(\supp(f))$, there exists $\lambda\in\Lambda$ such that $\sum_{u\in S_k}\varepsilon(k,u,\lambda)\leqslant\beta_k$ for all $k\in[1,n]$.

{\bf{(2)}}\,\,Suppose that $\prod_{i=1}^{n}S_{i}\neq\emptyset$, and $\mathbf{0}\not\in B_a$ for all $a\in\prod_{i=1}^{n}S_{i}$. Then, $(g(\lambda)\mid\lambda\in\Lambda)$ is a Gr\"{o}bner basis of $Q$ if and only if $\zeta_1=\zeta_2$, and for any $k\in[1,n]$, $S_{k}$ satisfies Condition (D) in $R$.
\end{theorem}

\begin{proof}
We note that (1) follows from (1) of Theorem 4.1, Theorem 4.2 and (5.2), and (2) follows from Lemma 5.1 and Theorem 4.1.
\end{proof}

\setlength{\parindent}{0em}
\begin{remark}
In [29, Theorem 6.12], Mezei derives his Nullstellensatz by giving a sufficient condition for $\zeta_1=\zeta_2$ to hold true (see [26, Lemma 6.11]). Therefore (1) of Theorem 5.1 extends Mezei's result from polynomials over a field to polynomials over a commutative ring; and from the conclusions of the form similar to (5) of Theorem 3.1 to the conclusions of the form similar to (3), (4) of Theorem 3.1; and to more general $g(\lambda)$'s. By contrast, we do not present explicit sufficient conditions for $\zeta_1=\zeta_2$ to hold true.
\end{remark}

\subsection{A common generalization of the Nullstellensatz in \cite{23,24} and \cite{9}}
\setlength{\parindent}{2em}
In this subsection, we give a common generalization of the Nullstellensatz for multisets established in K\'{o}s, R\'{o}nyai and M\'{e}sz\'{a}ros \cite{23,24} and the Nullstellensatz with multiplicity established in Ball and Serra \cite{9}.

\setlength{\parindent}{0em}
\begin{notation}
Let $(g_1,\dots,g_n)\in\Omega^{n}$ be a family of monic polynomials such that $g_k\in R[x_{k}]$ for all $k\in[1,n]$. Then, for any $t\in\mathbb{N}$, we let
\begin{equation}\mbox{$\mathcal{I}_t(g_1,\dots,g_n)=\langle\{\prod_{k=1}^{n}{g_k}^{\alpha_{k}}\mid\alpha\in \mathbb{N}^{n},\sum_{i=1}^{n}\alpha_i=t\}\rangle$}.\end{equation}
\end{notation}

\setlength{\parindent}{2em}
We collect some basic properties in the following proposition.

\setlength{\parindent}{0em}
\begin{proposition}
Let $(g_1,\dots,g_n)\in\Omega^{n}$ be a family of monic polynomials such that $g_k\in R[x_{k}]$ for all $k\in[1,n]$. Also fix $t\in\mathbb{N}$. Then, the following two statements hold:

{\bf{(1)}}\,\,$(\prod_{k=1}^{n}{g_k}^{\alpha_{k}}\mid\alpha\in \mathbb{N}^{n},\sum_{i=1}^{n}\alpha_i=t)$ is a Gr\"{o}bner basis of $\mathcal{I}_t(g_1,\dots,g_n)$;

{\bf{(2)}}\,\,Let $f\in\Omega$. Then, there uniquely exists $\psi\in \Omega$ such that $f-\psi\in \mathcal{I}_{t}(g_1,\dots,g_n)$, and for any $\beta\in\supp(\psi)$ and $\alpha\in \mathbb{N}^{n}$ with $\sum_{i=1}^{n}\alpha_{i}=t$, there exists $s\in[1,n]$ with $\beta_{s}\leqslant\deg(g_s)\alpha_{s}-1$. Moreover, such a $\psi$ satisfies that $\supp(\psi)\subseteq\Delta(\supp(f))$.
\end{proposition}

\begin{proof}
We note that (1) follows from (2) of Lemma 2.7, and so we only prove (2). To this end, let
$$\mbox{$D=\{(\deg(g_1)\alpha_1,\dots,\deg(g_n)\alpha_n)\mid \alpha\in \mathbb{N}^{n},\sum_{i=1}^{n}\alpha_i=t\}$}.$$
From (1) of Lemma 2.7, Lemma 3.1, (1) and Theorem 3.1, we deduce that $\Omega$ is the direct sum of $\mathcal{I}_t(g_1,\dots,g_n)$ and $\delta(\mathbb{N}^{n}-\nabla(D))$, which further implies the existence and uniqueness of $\psi$. Moreover, it follows from Lemma 3.1 that $\supp(\psi)\subseteq\Delta(\supp(f))$, as desired.
\end{proof}

\setlength{\parindent}{2em}
Now we state and prove the main result of this subsection. To this end, for any $k\in[1,n]$, we let $S_k$ be a finite subset of $R$, and let $(S_k,\psi_k)$ be a multiset. Define $(g_1,\dots,g_n)$ as
\begin{equation}\mbox{$\forall~k\in[1,n]:g_k=\prod_{u\in S_k}(x_k-u)^{\psi_k(u)}$}.\end{equation}

\setlength{\parindent}{0em}
\begin{theorem}
{\bf{(1)}}\,\,Suppose that for any $k\in[1,n]$, $S_{k}$ satisfies Condition (D) in $R$. Fix $t\in\mathbb{N}$. Consider an arbitrary $f\in\Omega$. Then, $f\in\mathcal{I}_t(g_1,\dots,g_n)$ if and only if for any $a\in\prod_{i=1}^{n}S_i$ and $\beta\in\mathbb{N}^{n}$ with $\sum_{i=1}^{n}\left\lfloor\frac{\beta_i}{\psi_i(a_i)}\right\rfloor\leqslant t-1$, it holds that $f(x_1+a_1,\dots,x_n+a_n)_{[\beta]}=0$.

{\bf{(2)}}\,\,Suppose that for any $k\in[1,n]$, $S_{k}$ satisfies Condition (D) in $R$. Fix $t\in\mathbb{N}$. Let $f\in\Omega$ such that for any $a\in\prod_{i=1}^{n}S_i$ and $\beta\in\mathbb{N}^{n}$ with $\sum_{i=1}^{n}\left\lfloor\frac{\beta_i}{\psi_i(a_i)}\right\rfloor\leqslant t-1$, it holds that $f(x_1+a_1,\dots,x_n+a_n)_{[\beta]}=0$. Then, there exists polynomials $(p(\alpha)\mid\alpha\in\mathbb{N}^{n},\sum_{i=1}^{n}\alpha_i=t)$ such that
$$\mbox{$f=\sum_{(\alpha\in\mathbb{N}^{n},\sum_{i=1}^{n}\alpha_i=t)}p(\alpha)\cdot(\prod_{k=1}^{n}{g_k}^{\alpha_k})$},$$
and for any $\alpha\in\mathbb{N}^{n}$ with $\sum_{i=1}^{n}\alpha_i=t$, it holds that $$\mbox{$\supp(p(\alpha))+\supp(\prod_{k=1}^{n}{g_k}^{\alpha_k})\subseteq\Delta(\supp(f))$}.$$
Moreover, for any $\beta\in\max(\supp(f))$, there exists $\gamma\in\mathbb{N}^{n}$, $\sum_{i=1}^{n}\gamma_i=t$ such that $(\sum_{u\in S_k}\psi_k(u))\cdot\gamma_k\leqslant\beta_k$ for all $k\in[1,n]$.

{\bf{(3)}}\,\,Suppose that $\prod_{i=1}^{n}S_{i}\neq\emptyset$. Let $t\in\mathbb{Z}^{+}$. Assume in addition that for any $f\in\Omega$ satisfying that
$$\mbox{$\forall~a\in\prod_{i=1}^{n}S_i,\forall~\beta\in\mathbb{N}^{n}~s.t.~\sum_{i=1}^{n}\left\lfloor\frac{\beta_i}{\psi_i(a_i)}\right\rfloor\leqslant t-1:f(x_1+a_1,\dots,x_n+a_n)_{[\beta]}=0$},$$
it holds that $f\in\mathcal{I}_t(g_1,\dots,g_n)$. Then, for any $k\in[1,n]$, $S_{k}$ satisfies Condition (D) in $R$.
\end{theorem}

\begin{proof}
Throughout the proof, we fix $t\in\mathbb{N}$, and let
$$\mbox{$\Lambda=\{\theta\in\mathbb{N}^{n}\mid\sum_{i=1}^{n}\theta_i=t\}$}.$$
Consider the tuple $(\psi_i(u)\theta_i\mid i\in[1,n],u\in S_i,\theta\in\Lambda)$. Then, we have
\begin{equation}\mbox{$\forall~\theta\in\Lambda:\prod_{i=1}^{n}\prod_{u\in S_i}(x_i-u)^{\psi_i(u)\theta_i}=\prod_{k=1}^{n}{g_k}^{\theta_k}$}.\end{equation}
Now we define $D$, $(B_a\mid a\in\prod_{i=1}^{n}S_i)$ and $Q$ as follows:
$$D=\mbox{$\{((\sum_{u\in S_1}\psi_1(u))\theta_1,\dots,(\sum_{u\in S_n}\psi_n(u))\theta_n)\mid \theta\in\Lambda\}$},$$
$$\mbox{$\forall~a\in\prod_{i=1}^{n}S_i:B_a=\{(\psi_1(a_1)\theta_1,\dots,\psi_n(a_n)\theta_n)\mid \theta\in\Lambda\}$},$$
$$Q\triangleq\{f\in\Omega\mid\mbox{$\supp(f(x_1+a_1,\dots,x_n+a_n))\subseteq\nabla(B_a)$ for all $a\in\prod_{i=1}^{n}S_i$}\}.$$
By Lemma 2.5, $\mathbb{N}^{n}-\nabla(D)$ is finite. From Section 5.1 and (5.9), we deduce that $Q$ is an ideal of $\Omega$ with $\mathcal{I}_t(g_1,\dots,g_n)\subseteq Q$. Moreover, for an arbitrary $f\in\Omega$, by Lemma 2.5, one can check that $f\in Q$ if and only if for any $a\in\prod_{i=1}^{n}S_i$ and $\beta\in\mathbb{N}^{n}$ with $\sum_{i=1}^{n}\left\lfloor\frac{\beta_i}{\psi_i(a_i)}\right\rfloor\leqslant t-1$, it holds that
$$f(x_1+a_1,\dots,x_n+a_n)_{[\beta]}=0.$$

{\bf{(1)}} and {\bf{(2)}}\,\,By Lemma 2.5, we have
\begin{eqnarray*}
\begin{split}
\mbox{$\sum_{a\in\prod_{i=1}^{n}S_i}|\mathbb{N}^{n}-\nabla(B_a)|$}&=\mbox{$\sum_{a\in\prod_{i=1}^{n}S_i}\left(\prod_{i=1}^{n}\psi_i(a_i)\right)\cdot\binom{n+t-1}{n}$}\\
&=\mbox{$(\prod_{i=1}^{n}(\sum_{u\in S_i}\psi_i(u)))\cdot\binom{n+t-1}{n}$}\\
&=|\mathbb{N}^{n}-\nabla(D)|.
\end{split}
\end{eqnarray*}
Hence by (1) of Theorem 5.1 and (5.9), $(\prod_{k=1}^{n}{g_k}^{\theta_k}\mid\theta\in\Lambda)$ is a Gr\"{o}bner basis of $Q$. It then follows from Theorem 3.1 that $Q=\mathcal{I}_t(g_1,\dots,g_n)$, which further establishes (1). Moreover, (2) again follows from (1) of Theorem 5.1.

{\bf{(3)}}\,\,Assume that $t\in\mathbb{Z}^{+}$ and $\prod_{i=1}^{n}S_{i}\neq\emptyset$. Then, one can check that $\mathbf{0}\not\in B_a$ for all $a\in\prod_{i=1}^{n}S_{i}$. Moreover, we note that $Q=\mathcal{I}_t(g_1,\dots,g_n)$. It then follows from (1) of Proposition 5.1 that $(\prod_{k=1}^{n}{g_k}^{\theta_k}\mid\theta\in\Lambda)$ is a Gr\"{o}bner basis of $Q$. Hence an application of (2) of Theorem 5.1 completes the proof.
\end{proof}

\setlength{\parindent}{2em}
We note that by Theorem 5.2, Condition (D) is a necessary and sufficient condition for our Nullstellensatz to hold true under the relatively mild assumptions that $t\geqslant1$ and $\prod_{i=1}^{n}S_{i}\neq\emptyset$. In the following remark, we show how Theorem 5.2 recovers some known results in the literature.

\setlength{\parindent}{0em}
\begin{remark}
First, if $R$ is a field and $\psi_i(u)=1$ for all $i\in[1,n]$ and $u\in S_i$, then (2) of Theorem 5.2 recovers the Nullstellensatz with multiplicity [9, Theorem 3.1], and enhances [9, Corollary 3.2]; if $t=1$, then (2) of Theorem 5.2 enhances [23, Theorem 4], which further recovers [24, Theorem 6] when $R$ is a field; moreover, if $R$ is a field and $t=1$, then (2) of Theorem 5.2, together with (2) of Proposition 5.1, recovers the Nullstellensatz for multisets [24, Theorem 1].

\hspace*{4mm}\,\,Second, if $t=1$ and $\psi_i(u)=1$ for all $i\in[1,n]$ and $u\in S_i$, then Theorem 5.2 recovers the Chevalley-Alon-Tarsi-Schauz Lemma (see, [34, Equivalence and Definition 2.8], [12, Theorem 2.5], [17, Theorem 3.3], [18, Theorem 2.3]) and [17, Theorem 3.8]. Both of these two results have shown that Alon's combinatorial Nullstellensatz Theorem 1.1 remains valid with the field $\mathbb{F}$ replaced by $R$ if and only if all the $S_{i}$'s satisfy Condition (D) in $R$.
\end{remark}

\subsection{Punctured Nullstellensatz}
\setlength{\parindent}{2em}
In this subsection, we give a common generalization of the punctured Nullstellensatz established in Ball and Serra \cite{9}, and in K\'{o}s, R\'{o}nyai and M\'{e}sz\'{a}ros \cite{24}. We begin with the following proposition, which is inspired by [9, Theorem 4.1] and [24, Theorem 7].

\setlength{\parindent}{0em}
\begin{proposition}
Let $(g_1,\dots,g_n)\in\Omega^{n}$, $(h_1,\dots,h_n)\in\Omega^{n}$ satisfy that $g_k,h_k\in R[x_{k}]$, $g_k,h_k$ are monic, $h_k\mid g_k$ for all $k\in[1,n]$. Let $t\in\mathbb{N}$, and fix $f\in\bigcap_{m=1}^{n}\mathcal{I}_{t}(g_1,\dots,g_{m-1},g_m/h_m,g_{m+1},\dots,g_n)$. Moreover, let $\psi\in \Omega$ such that $f-\psi\in \mathcal{I}_{t}(g_1,\dots,g_n)$, and for any $\beta\in\supp(\psi)$ and $\alpha\in \mathbb{N}^{n}$ with $\sum_{i=1}^{n}\alpha_{i}=t$, there exists $s\in[1,n]$ with $\beta_{s}\leqslant\deg(g_s)\alpha_{s}-1$. Then, the following three statements hold:

{\bf{(1)}}\,\,$\psi\in\bigcap_{m=1}^{n}\mathcal{I}_{t}(g_1,\dots,g_{m-1},g_m/h_m,g_{m+1},\dots,g_n)$;

{\bf{(2)}}\,\,$(\prod_{k=1}^{n}\frac{g_k}{h_k})\mid\psi$;

{\bf{(3)}}\,\,Suppose that $t\in\mathbb{Z}^{+}$. Let $\varphi\in\Omega$ with $\psi=\varphi\cdot(\prod_{k=1}^{n}\frac{g_k}{h_k})$, and let $(u_{1},\dots,u_{n})\in R^{n}$ such that $g_k(u_{1},\dots,u_{n})=0$ for all $k\in[1,n]$. Moreover, fix $m\in[1,n]$, and let $a=\prod_{k\in[1,n]-\{m\}}(g_k/h_k)(u_{1},\dots,u_{n})$. Then, we have
\begin{equation}(g_{m}/h_{m})^{t-1}\mid a\cdot\varphi(u_{1},\dots,u_{m-1},x_{m},u_{m+1},\dots,u_{n}).\end{equation}
Assume in addition that $f(u_{1},\dots,u_{n})\neq0$, and $(g_k/h_k)(u_{1},\dots,u_{n})$ is not a zero divisor of $R$ for any $k\in[1,n]$. Then, it holds that
\begin{equation}\mbox{$\deg(f)\geqslant\deg(\psi)\geqslant(t-1)\deg(g_m/h_m)+\left(\sum_{k=1}^{n}\deg(g_k/h_k)\right)$}.\end{equation}
\end{proposition}

\begin{proof}
{\bf{(1)}}\,\,The desired result follows from the fact that $\mathcal{I}_t(g_1,\dots,g_n)\subseteq\bigcap_{m=1}^{n}\mathcal{I}_{t}(g_1,\dots,g_{m-1},g_m/h_m,g_{m+1},\dots,g_n)$.

{\bf{(2)}}\,\,Let $m\in[1,n]$, and let $(\zeta_1,\dots,\zeta_n)=(g_1,\dots,g_{m-1},g_m/h_m,g_{m+1},\dots,g_n)$. By (1), (1) of Proposition 5.1 and Theorem 3.1, we can choose polynomials $(p(\alpha)\mid\alpha\in \mathbb{N}^{n},\sum_{i=1}^{n}\alpha_i=t)$ satisfying the following two conditions:
\begin{equation}\mbox{$\psi=\sum_{(\alpha\in\mathbb{N}^{n},\sum_{i=1}^{n}\alpha_i=t)}p(\alpha)\cdot(\prod_{k=1}^{n}{\zeta_{k}}^{\alpha_{k}})$},\end{equation}
\begin{equation}\hspace*{-14mm}\mbox{$\forall~\alpha\in \mathbb{N}^{n}~s.t.~\sum_{i=1}^{n}\alpha_i=t:\supp(p(\alpha))+\supp(\prod_{k=1}^{n}{\zeta_{k}}^{\alpha_{k}})\subseteq\Delta(\supp(\psi))$}.\end{equation}
Consider an arbitrary $\alpha\in \mathbb{N}^{n}$ such that $\sum_{i=1}^{n}\alpha_i=t$, $p(\alpha)\neq0$. By (5.13), we can choose $\beta\in\supp(\psi)$ with $(\deg(\zeta_1)\alpha_1,\dots,\deg(\zeta_n)\alpha_n)\leqslant\beta$. Now we further choose $s\in[1,n]$ with $\beta_{s}\leqslant\deg(g_s)\alpha_{s}-1$. It then follows that $\zeta_s\neq g_s$, and hence $s=m$, which further implies that $\alpha_m\geqslant1$. It then follows from (5.12) that $(g_m/h_m)\mid\psi$. Now (2) immediately follows from the arbitrariness of $m$ and Lemma 2.4.

{\bf{(3)}}\,\,Let $(\zeta_1,\dots,\zeta_n)=(g_1,\dots,g_{m-1},g_m/h_m,g_{m+1},\dots,g_n)$. By (1), we can choose polynomials $(p(\alpha)\mid\alpha\in \mathbb{N}^{n},\sum_{i=1}^{n}\alpha_i=t)$ satisfying (5.12). Let $\beta\in\mathbb{N}^{n}$ such that $\beta_m=t$ and $\beta_i=0$ for all $i\in[1,n]-\{m\}$. Consider an arbitrary $\alpha\in \mathbb{N}^{n}$ with $\sum_{i=1}^{n}\alpha_i=t$, $\alpha\neq\beta$. Then, there exists $s\in[1,n]-\{m\}$ with $\alpha_s\geqslant1$. It follows that $g_s\mid (\prod_{k=1}^{n}{\zeta_{k}}^{\alpha_{k}})$. Since $g_s\in R[x_s]$, $s\neq m$, we have $g_s(u_1,\dots,u_{m-1},x_m,u_{m+1},\dots,u_{n})=g_s(u_1,\dots,u_n)=0$, which further implies that $(\prod_{k=1}^{n}{\zeta_{k}}^{\alpha_{k}})(u_1,\dots,u_{m-1},x_m,u_{m+1},\dots,u_{n})=0$. From the above discussion, we deduce that
$$\hspace*{-6mm}\psi(u_1,\dots,u_{m-1},x_m,u_{m+1},\dots,u_{n})=(p(\beta))(u_1,\dots,u_{m-1},x_m,u_{m+1},\dots,u_{n})\cdot(g_m/h_m)^{t}.$$
On the other hand, it follows from $\psi=\varphi\cdot(\prod_{k=1}^{n}\frac{g_k}{h_k})$ that
$$\hspace*{-6mm}\mbox{$\psi(u_1,\dots,u_{m-1},x_m,u_{m+1},\dots,u_{n})=a\cdot\varphi(u_1,\dots,u_{m-1},x_m,u_{m+1},\dots,u_{n})\cdot (g_m/h_m)$}.$$
Combining the above two equations, (5.10) immediately follows from (4) of Lemma 2.1 and the fact that $g_m/h_m$ is monic.

\hspace*{4mm}\,\,Now we prove (5.11). Noting that $(f-\psi)(u_{1},\dots,u_{n})=0$, we have $\psi(u_{1},\dots,u_{n})=f(u_{1},\dots,u_{n})\neq0$. By $\varphi\mid\psi$, we have $\varphi(u_{1},\dots,u_{n})\neq0$, which, together with the fact that $a$ is not a zero divisor of $R$, further implies that $a\cdot\varphi(u_{1},\dots,u_{m-1},x_{m},u_{m+1},\dots,u_{n})\neq0$. Noticing that $(g_m/h_m)^{t-1}$ is a monic polynomial with $\deg((g_m/h_m)^{t-1})=(t-1)\deg(g_m/h_m)$, by (5.10) and (4) of Lemma 2.1, we have
$$(t-1)\deg(g_m/h_m)\leqslant\deg(a\cdot\varphi(u_{1},\dots,u_{m-1},x_{m},u_{m+1},\dots,u_{n}))\leqslant\deg(\varphi).$$
Since $g_k/h_k$ is monic for all $k\in[1,n]$, by (4) of Lemma 2.1, we have
$$\hspace*{-6mm}\mbox{$\deg(\psi)=\deg(\varphi)+(\sum_{k=1}^{n}\deg(g_k/h_k))\geqslant(t-1)\deg(g_m/h_m)+(\sum_{k=1}^{n}\deg(g_k/h_k))$}.$$
Finally, by (2) of Proposition 5.1, we have $\deg(\psi)\leqslant\deg(f)$, which further establishes (5.11), as desired.
\end{proof}

\setlength{\parindent}{2em}
Now we state and prove our punctured Nullstellensatz. Throughout the rest of this subsection, for any $k\in[1,n]$, we let $S_k$ be a finite subset of $R$, $(S_k,\psi_k)$ be a multiset, and fix $E_k\subseteq S_k$. Define $(g_1,\dots,g_n)$ and $(h_1,\dots,h_n)$ as follows:
\begin{equation}\mbox{$\forall~k\in[1,n]:g_k=\prod_{u\in S_k}(x_k-u)^{\psi_k(u)},~h_k=\prod_{u\in E_k}(x_k-u)^{\psi_k(u)}$}.\end{equation}

\setlength{\parindent}{0em}
\begin{theorem}
Suppose that for any $k\in[1,n]$, $S_{k}$ satisfies Condition (D) in $R$. Let $t\in\mathbb{N}$, $f\in\Omega$. Then, it holds that:

{\bf{(1)}}\,\,$f\in\bigcap_{m=1}^{n}\mathcal{I}_t(g_1,\dots,g_{m-1},g_m/h_m,g_{m+1},\dots,g_n)$ if and only if for any $a\in(\prod_{i=1}^{n}S_i)-(\prod_{i=1}^{n}E_i)$ and $\beta\in\mathbb{N}^{n}$ with $\sum_{i=1}^{n}\left\lfloor\frac{\beta_i}{\psi_i(a_i)}\right\rfloor\leqslant t-1$, it holds that $f(x_1+a_1,\dots,x_n+a_n)_{[\beta]}=0$;

{\bf{(2)}}\,\,Suppose that for any $a\in(\prod_{i=1}^{n}S_i)-(\prod_{i=1}^{n}E_i)$ and $\beta\in\mathbb{N}^{n}$ with $\sum_{i=1}^{n}\left\lfloor\frac{\beta_i}{\psi_i(a_i)}\right\rfloor\leqslant t-1$, it holds that
$$f(x_1+a_1,\dots,x_n+a_n)_{[\beta]}=0.$$
Then, there uniquely exists $\eta\in \Omega$ such that $f-\eta\in \mathcal{I}_{t}(g_1,\dots,g_n)$, and for any $\beta\in\supp(\eta)$ and $\alpha\in \mathbb{N}^{n}$ with $\sum_{i=1}^{n}\alpha_{i}=t$, there exists $k\in[1,n]$ with $\beta_{k}\leqslant(\sum_{u\in S_k}\psi_k(u))\alpha_{k}-1$. Moreover, we have $\supp(\eta)\subseteq\Delta(\supp(f))$, $(\prod_{k=1}^{n}\frac{g_k}{h_k})\mid\eta$. Assume in addition that $t\geqslant1$ and $f(w_{1},\dots,w_{n})\neq0$ for some $w\in\prod_{i=1}^{n}S_i$. Then, for any $m\in[1,n]$, it holds that
$$\mbox{$\deg(f)\geqslant\deg(\eta)\geqslant(t-1)(\sum_{u\in S_m-E_m}\psi_m(u))+(\sum_{k=1}^{n}\sum_{u\in S_k-E_k}\psi_k(u))$}.$$
\end{theorem}

\begin{proof}
{\bf{(1)}}\,\,First, let $m\in[1,n]$. Define $(P_1,\dots,P_n)$ as $P_m=S_m-E_m$, and $P_k=S_k$ for all $k\in[1,n]-\{m\}$. Then, for any $k\in[1,n]$, $P_{k}$ satisfies Condition (D) in $R$. An application of (1) of Theorem 5.2 to $(P_1,\psi_1),\dots,(P_n,\psi_n)$ yields the following fact: $f\in\mathcal{I}_t(g_1,\dots,g_{m-1},g_m/h_m,g_{m+1},\dots,g_n)$ if and only if for any $a\in \prod_{i=1}^{n}P_i$ and $\beta\in\mathbb{N}^{n}$ with $\sum_{i=1}^{n}\left\lfloor\frac{\beta_i}{\psi_i(a_i)}\right\rfloor\leqslant t-1$, it holds that $f(x_1+a_1,\dots,x_n+a_n)_{[\beta]}=0$. Now, with the above discussion, the desired result follows from the fact that
$$\hspace*{-4mm}\mbox{$(\prod_{i=1}^{n}S_i)-(\prod_{i=1}^{n}E_i)=\bigcup_{m=1}^{n}S_1\times\dots\times S_{m-1}\times (S_m-E_m)\times S_{m+1}\times\dots \times S_n$}.$$

{\bf{(2)}}\,\,By (1), $f\in\bigcap_{m=1}^{n}\mathcal{I}_t(g_1,\dots,g_{m-1},g_m/h_m,g_{m+1},\dots,g_n)$. Hence all the desired results except the last assertion follow from Propositions 5.1 and 5.2. Therefore in what follows, we assume that $t\geqslant1$, and fix $w\in\prod_{i=1}^{n}S_i$ with $f(w_{1},\dots,w_{n})\neq0$. We note that $w\in\prod_{i=1}^{n}E_i$. Indeed, if $w\in(\prod_{i=1}^{n}S_i)-(\prod_{i=1}^{n}E_i)$, then by $\sum_{i=1}^{n}\left\lfloor\frac{\mathbf{0}_i}{\psi_i(w_i)}\right\rfloor=0\leqslant t-1$, we have $f(w_1,\dots,w_n)=f(x_1+w_1,\dots,x_n+w_n)_{[\mathbf{0}]}=0$, a contradiction. For any $k\in[1,n]$, from $w_k\in E_k$ and $S_k$ satisfies Condition (D) in $R$, one can check that $g_k(w_1,\dots,w_n)=0$, and $(g_k/h_k)(w_1,\dots,w_n)$ is not a zero divisor of $R$. With these facts, the desired result follows from (3) of Proposition 5.2.
\end{proof}

\setlength{\parindent}{2em}
The following corollary immediately follows from (1) of Theorem 5.3 and (1) of Theorem 5.2.

\setlength{\parindent}{0em}
\begin{corollary}
Suppose that for any $k\in[1,n]$, $S_{k}$ satisfies Condition (D) in $R$. Then, for any $t\in\mathbb{N}$, it holds that
$$\hspace*{-2mm}\mbox{$(\bigcap_{m=1}^{n}\mathcal{I}_t(g_1,\dots,g_{m-1},g_m/h_m,g_{m+1},\dots,g_n))\cap\mathcal{I}_t(h_1,\dots,h_n)=\mathcal{I}_t(g_1,\dots,g_n)$}.$$
\end{corollary}

\setlength{\parindent}{0em}
\begin{remark}
{\bf{(1)}}\,\,If $\prod_{i=1}^{n}E_i=\emptyset$, then (1) of Theorem 5.3 recovers (1) of Theorem 5.2.

{\bf{(2)}}\,\,If $R$ is a field and $\psi_i(u)=1$ for all $i\in[1,n]$ and $u\in S_i$, then (2) of Theorem 5.3 recovers [9, Theorem 4.1]; if $R$ is a field and $t=1$, then (2) of Theorem 5.3 recovers [24, Theorem 7]; and if $t=1$ and $\psi_i(u)=1$ for all $i\in[1,n]$ and $u\in S_i$, then (2) of Theorem 5.3 recovers the punctured Nullstellensatz established in the proof of [18, Theorem 2.2].
\end{remark}

\subsection{A generalization of [33, Theorem 1.5]}
\setlength{\parindent}{2em}
In this subsection, we use Theorems 4.1 and 4.2 to derive another Nullstellensatz. Our Nullstellensatz Theorem 5.4 is inspired by Sauermann and Wigderson \cite{33}, where the authors study the minimum possible degree of a polynomial that vanishes to high order on most of the hypercube $\{0,1\}^{n}\subseteq\mathbb{R}^{n}$ (also see Clifton and Huang \cite{19}). As an application of Theorem 5.4, we will derive a generalization of [33, Theorem 1.5].

For any $k\in[1,n]$, let $S_k$ be a finite subset of $R$, $(S_k,\psi_k)$ be a multiset, and fix $E_k\subseteq S_k$. Define $(g_1,\dots,g_n)$ and $(h_1,\dots,h_n)$ as follows:
\begin{equation}\mbox{$\forall~k\in[1,n]:g_k=\prod_{u\in S_k}(x_k-u)^{\psi_k(u)},~h_k=\prod_{u\in E_k}(x_k-u)^{\psi_k(u)}$}.\end{equation}
We fix $t\in\mathbb{Z}^{+}$, and let
\begin{equation}\mbox{$\Lambda=\{\alpha\in\mathbb{N}^{n}\mid\sum_{i=1}^{n}\alpha_i\in\{t-1,t\}\}$}.\end{equation}
Define $(\varphi(\alpha)\mid\alpha\in\Lambda)\in\Omega^{\Lambda}$ as follows:
\begin{equation}\mbox{$\forall~\alpha\in\mathbb{N}^{n}~s.t.~\sum_{i=1}^{n}\alpha_i=t:\varphi(\alpha)=\prod_{i=1}^{n}{g_i}^{\alpha_i}$},\end{equation}
\begin{equation}\mbox{$\forall~\alpha\in\mathbb{N}^{n}~s.t.~\sum_{i=1}^{n}\alpha_i=t-1:\varphi(\alpha)=(\prod_{i=1}^{n}{g_i}^{\alpha_i})\cdot(\prod_{i=1}^{n}\frac{g_i}{h_i})$}.\end{equation}
Moreover, let $Q\subseteq\Omega$ such that for any $f\in\Omega$, $f\in Q$ if and only if the following two conditions hold:\\
{\bf{(i)}}\,\,For any $a\in(\prod_{i=1}^{n}S_i)-(\prod_{i=1}^{n}E_i)$ and $\beta\in\mathbb{N}^{n}$ with $\sum_{i=1}^{n}\left\lfloor\frac{\beta_i}{\psi_i(a_i)}\right\rfloor\leqslant t-1$, it holds that $f(x_1+a_1,\dots,x_n+a_n)_{[\beta]}=0$;\\
{\bf{(ii)}}\,\,For any $a\in\prod_{i=1}^{n}E_i$ and $\beta\in\mathbb{N}^{n}$ with $\sum_{i=1}^{n}\left\lfloor\frac{\beta_i}{\psi_i(a_i)}\right\rfloor\leqslant t-2$, it holds that $f(x_1+a_1,\dots,x_n+a_n)_{[\beta]}=0$.

\setlength{\parindent}{0em}
\begin{theorem}
{\bf{(1)}}\,\,$Q$ is an ideal of $\Omega$, and $(\varphi(\alpha)\mid\alpha\in\Lambda)$ is a family of monic polynomials with $\{\varphi(\alpha)\mid\alpha\in\Lambda\}\subseteq Q$.

{\bf{(2)}}\,\,Suppose that for any $k\in[1,n]$, $S_{k}$ satisfies Condition (D) in $R$. Then, $(\varphi(\alpha)\mid\alpha\in\Lambda)$ is a Gr\"{o}bner basis of $Q$. Moreover, for an arbitrary $f\in Q$, there exists $(p(\alpha)\mid\alpha\in\Lambda)\in\Omega^{\Lambda}$ such that $f=\sum_{\alpha\in\Lambda}p(\alpha)\cdot\varphi(\alpha)$ and $$\mbox{$\forall~\alpha\in\Lambda:\supp(p(\alpha))+\supp(\varphi(\alpha))\subseteq\Delta(\supp(f))$}.$$

{\bf{(3)}}\,\,Suppose that $\prod_{i=1}^{n}S_{i}\neq\emptyset$, $t\geqslant2$, $(\varphi(\alpha)\mid\alpha\in\Lambda)$ is a Gr\"{o}bner basis of $Q$. Then, for any $k\in[1,n]$, $S_{k}$ satisfies Condition (D) in $R$.
\end{theorem}

\begin{proof}
Since (1) follows from some straightforward verification, in what follows, we only prove (2) and (3). Define $(B_a\mid a\in\prod_{i=1}^{n}S_{i})$ as follows:
$$\hspace*{-12mm}\mbox{$\forall~a\in(\prod_{i=1}^{n}S_{i})-(\prod_{i=1}^{n}E_{i}):B_{a}=\{(\psi_1(a_{1})\alpha_1,\dots,\psi_n(a_{n})\alpha_n)\mid\alpha\in\mathbb{N}^{n},\sum_{i=1}^{n}\alpha_i=t\}$},$$
$$\mbox{$\forall~a\in\prod_{i=1}^{n}E_{i}:B_{a}=\{(\psi_1(a_{1})\alpha_1,\dots,\psi_n(a_{n})\alpha_n)\mid\alpha\in\mathbb{N}^{n},\sum_{i=1}^{n}\alpha_i=t-1\}$}.$$
By Lemma 2.5, for any $a\in\prod_{i=1}^{n}S_{i}$, $\mathbb{N}^{n}-\nabla(B_{a})$ is finite. Moreover, Lemma 2.5 implies that
$$Q=\{f\in\Omega\mid \mbox{$\supp(f(x_{1}+a_{1},\dots,x_{n}+a_{n}))\subseteq\nabla(B_{a})$ for all $a\in\prod_{i=1}^{n}S_{i}$}\}.$$
Again by Lemma 2.5, we have
\begin{eqnarray*}
\begin{split}
&\mbox{$\sum_{a\in\prod_{i=1}^{n}S_i}|\mathbb{N}^{n}-\nabla(B_a)|$}\\
&=\mbox{$(\sum_{a\in(\prod_{i=1}^{n}S_i)-(\prod_{i=1}^{n}E_i)}\prod_{i=1}^{n}\psi_i(a_i))\cdot\binom{n+t-1}{n}+(\sum_{a\in\prod_{i=1}^{n}E_i}\prod_{i=1}^{n}\psi_i(a_i))\cdot\binom{n+t-2}{n}$}\\
&=\mbox{$(\sum_{a\in\prod_{i=1}^{n}S_i}\prod_{i=1}^{n}\psi_i(a_i))\binom{n+t-1}{n}-(\sum_{a\in\prod_{i=1}^{n}E_i}\prod_{i=1}^{n}\psi_i(a_i))(\binom{n+t-1}{n}-\binom{n+t-2}{n})$}\\
&=\mbox{$(\prod_{i=1}^{n}(\sum_{u\in S_i}\psi_i(u)))\binom{n+t-1}{n}-(\prod_{i=1}^{n}(\sum_{u\in E_i}\psi_i(u)))\binom{n+t-2}{n-1}$}\\
&=\mbox{$(\prod_{i=1}^{n}\deg(g_i))\binom{n+t-1}{n}-(\prod_{i=1}^{n}\deg(h_i))\binom{n+t-2}{n-1}$}.
\end{split}
\end{eqnarray*}
Now define $(\theta(\alpha)\mid\alpha\in\Lambda)$ as follows:
$$\mbox{$\forall~\alpha\in\mathbb{N}^{n}~s.t.~\sum_{i=1}^{n}\alpha_i=t:\theta(\alpha)=(\deg(g_1)\alpha_1,\dots,\deg(g_n)\alpha_n)$},$$
$$\mbox{$\forall~\alpha\in\mathbb{N}^{n}~s.t.~\sum_{i=1}^{n}\alpha_i=t-1:\theta(\alpha)=(\deg(g_i)\alpha_i+\deg(g_i)-\deg(h_i)\mid i\in[1,n])$},$$
and let $D=\{\theta(\alpha)\mid\alpha\in\Lambda\}$. Then, one can check that for any $\alpha\in\Lambda$, $\theta(\alpha)$ is the greatest element of $\supp(\varphi(\alpha))$. Moreover, by Lemma 2.6, we have
$$\hspace*{-6mm}\mbox{$|\mathbb{N}^{n}-\nabla(D)|=(\prod_{i=1}^{n}\deg(g_i))\binom{n+t-1}{n}-(\prod_{i=1}^{n}\deg(h_i))\binom{n+t-2}{n-1}=\sum_{a\in\prod_{i=1}^{n}S_{i}}|\mathbb{N}^{n}-\nabla(B_{a})|$}.$$
With the above discussion, (2) immediately follows from Theorems 4.1 and 4.2. Finally, noticing that if $t\geqslant2$, then $\mathbf{0}\not\in B_a$ for all $a\in\prod_{i=1}^{n}S_{i}$, (3) immediately follows from Theorem 4.1, as desired.
\end{proof}

\setlength{\parindent}{2em}
Theorem 5.4, together with Theorems 5.2 and 5.3, implies the following corollary.

\setlength{\parindent}{0em}
\begin{corollary}
Suppose that for any $k\in[1,n]$, $S_{k}$ satisfies Condition (D) in $R$, and let $e=\min\{\sum_{u\in S_k}\psi_k(u)\mid k\in[1,n]\}$. Then, the following four statements hold:

{\bf{(1)}}\,\,$Q=(\bigcap_{m=1}^{n}\mathcal{I}_t(g_1,\dots,g_{m-1},g_m/h_m,g_{m+1},\dots,g_n))\cap\mathcal{I}_{t-1}(h_1,\dots,h_n)$;

{\bf{(2)}}\,\,$Q=\mathcal{I}_t(g_1,\dots,g_n)+\mathcal{I}_{t-1}(g_1,\dots,g_n)\cdot(\prod_{i=1}^{n}\frac{g_i}{h_i})$;

{\bf{(3)}}\,\,For any $f\in Q$, $f\not\in\mathcal{I}_t(g_1,\dots,g_n)$ if and only if there exists $v\in\prod_{i=1}^{n}E_i$ and $\gamma\in\mathbb{N}^{n}$ such that $\sum_{i=1}^{n}\left\lfloor\frac{\gamma_i}{\psi_i(v_i)}\right\rfloor=t-1$, $f(x_1+v_1,\dots,x_n+v_n)_{[\gamma]}\neq0$;

{\bf{(4)}}\,\,Suppose that $\prod_{i=1}^{n}E_i\neq\emptyset$. Then, for any $\gamma\in\mathbb{N}^{n}$ with $\sum_{i=1}^{n}\gamma_i=t-1$, we have $\varphi(\gamma)\in Q-\mathcal{I}_t(g_1,\dots,g_n)$. Moreover, it holds that
$$\mbox{$\min\{\deg(f)\mid f\in Q-\mathcal{I}_t(g_1,\dots,g_n)\}=(t-1)e+(\sum_{i=1}^{n}\sum_{u\in S_i-E_i}\psi_i(u))$}.$$
\end{corollary}

\begin{proof}
With the definition of $Q$, (1) follows from (1) of Theorem 5.3 and (1) of Theorem 5.2, (2) follows from (1), (2) of Theorem 5.4, and (3) follows from (1) of Theorem 5.2. Therefore it remains to establish (4). To this end, we first fix $\gamma\in\mathbb{N}^{n}$ with $\sum_{i=1}^{n}\gamma_i=t-1$. Noticing that $\varphi(\gamma)\in Q$, it suffices to show that $\varphi(\gamma)\not\in\mathcal{I}_t(g_1,\dots,g_n)$. By way of contradiction, assume that $\varphi(\gamma)\in\mathcal{I}_t(g_1,\dots,g_n)$. Since $(\deg(g_i)\gamma_i+\deg(g_i)-\deg(h_i)\mid i\in[1,n])$ is the greatest element of $\supp(\varphi(\gamma))$, by Theorem 5.2, there exists $\alpha\in\mathbb{N}^{n}$ such that $\sum_{i=1}^{n}\alpha_i=t$, and $\deg(g_i)\alpha_i\leqslant\deg(g_i)\gamma_i+\deg(g_i)-\deg(h_i)$ for all $i\in[1,n]$. Consider an arbitrary $k\in[1,n]$. Since $E_k\neq\emptyset$, we have $\deg(g_k)\geqslant\deg(h_k)\geqslant1$. This, together with $\deg(g_k)\alpha_k\leqslant\deg(g_k)\gamma_k+\deg(g_k)-\deg(h_k)$, implies that $\alpha_k\leqslant\gamma_k$. It then follows that $\alpha\leqslant\gamma$, $\sum_{i=1}^{n}\alpha_i=t$, $\sum_{i=1}^{n}\gamma_i=t-1$, a contradiction, as desired.

\hspace*{4mm}\,\,Next, choose $l\in[1,n]$ with $\deg(g_l)=e$. By the previous paragraph, we have ${g_l}^{t-1}(\prod_{i=1}^{n}g_i/h_i)\in Q-\mathcal{I}_t(g_1,\dots,g_n)$. Moreover, one can check that
$$\mbox{$\deg({g_l}^{t-1}(\prod_{i=1}^{n}g_i/h_i))=(t-1)e+(\sum_{i=1}^{n}\sum_{u\in S_i-E_i}\psi_i(u))$}.$$

\hspace*{4mm}\,\,Finally, consider an arbitrary $f\in Q-\mathcal{I}_t(g_1,\dots,g_n)$. By Theorem 5.4, there exists $(p(\alpha)\mid\alpha\in\Lambda)\in\Omega^{\Lambda}$ such that $f=\sum_{\alpha\in\Lambda}p(\alpha)\cdot\varphi(\alpha)$, and $\deg(p(\alpha))+\deg(\varphi(\alpha))\leqslant\deg(f)$ for all $\alpha\in\Lambda$. Since $f\not\in\mathcal{I}_t(g_1,\dots,g_n)$, we can choose $\beta\in\mathbb{N}^{n}$ such that $\sum_{i=1}^{n}\beta_i=t-1$, $p(\beta)\neq0$. It then follows that
\begin{eqnarray*}
\begin{split}
\deg(f)\geqslant\deg(\varphi(\beta))&=\mbox{$(\sum_{i=1}^{n}\deg(g_i)\beta_i)+(\sum_{i=1}^{n}\sum_{u\in S_i-E_i}\psi_i(u))$}\\
&\geqslant\mbox{$(t-1)e+(\sum_{i=1}^{n}\sum_{u\in S_i-E_i}\psi_i(u))$},
\end{split}
\end{eqnarray*}
which further completes the proof of (4).
\end{proof}

\setlength{\parindent}{0em}
\begin{remark}
{\bf{(1)}}\,\,If $\prod_{i=1}^{n}E_i=\emptyset$, then (1), (2) of Theorem 5.4 recover (1), (2) of Theorem 5.2.

{\bf{(2)}}\,\,If $R$ is a field, $S_i=\{0,1_R\}$, $E_i=\{0\}$, $\psi_i(0)=\psi_i(1_R)=1$ for all $i\in[1,n]$, then (3), (4) of Corollary 5.2 recover [33, Theorem 1.5].
\end{remark}

\section{Applications of our punctured Nullstellensatz}

\subsection{Hyperplane covering}

\setlength{\parindent}{2em}
Throughout this subsection, for any $k\in[1,n]$, we let $S_k$ be a finite subset of $R$, $(S_k,\psi_k)$ be a multiset, and fix $E_k\subseteq S_k$.

Roughly speaking, given a family of $r$-hyperplanes (see Definition 2.5), we consider the case that every point of $(\prod_{i=1}^{n}S_i)-(\prod_{i=1}^{n}E_i)$ is covered by many $r$-hyperplanes, yet at least one point of $\prod_{i=1}^{n}S_i$ is not covered by any $r$-hyperplane. The following is the most general result of this section.

\setlength{\parindent}{0em}
\begin{theorem}
Suppose that for any $k\in[1,n]$, $S_{k}$ satisfies Condition (D) in $R$. Let $L$ be a finite set, $(e_\lambda\mid\lambda\in L)\in\mathbb{N}^{L}$, and fix $(H_\lambda\mid\lambda\in L)$ such that for any $\lambda\in L$, $H_\lambda$ is an $e_\lambda$-hyperplane of $R^{n}$. Moreover, let $(\rho_{\lambda}\mid\lambda\in L)\in\Omega^{L}$ such that for any $\lambda\in L$, it holds that $\deg(\rho_\lambda)=e_\lambda$ and
$$H_\lambda=\{u\in R^{n}\mid \rho_\lambda(u_1,\dots,u_n)=0\}.$$
Fix $t\in \mathbb{Z}^{+}$. Assume in addition that the following two conditions hold:

{\bf{(i)}}\,\,For any $a\in(\prod_{i=1}^{n}S_i)-(\prod_{i=1}^{n}E_i)$, it holds that
$$\mbox{$|\{\lambda\in L\mid a\in H_\lambda\}|\geqslant\max\{\sum_{i=1}^{n}\beta_i\mid\beta\in \mathbb{N}^{n},\sum_{i=1}^{n}\left\lfloor\frac{\beta_i}{\psi_i(a_i)}\right\rfloor\leqslant t-1\}+1$};$$
{\bf{(ii)}}\,\,There exists $w\in\prod_{i=1}^{n}S_i$ such that $\prod_{\lambda\in L}\rho_{\lambda}(w_1,\dots,w_n)\neq0$.

Then, for any $m\in[1,n]$, it holds that
$$\hspace*{-6mm}\mbox{$\sum_{\lambda\in L}e_{\lambda}\geqslant\deg(\prod_{\lambda\in L}\rho_{\lambda})\geqslant(t-1)(\sum_{u\in S_m-E_m}\psi_m(u))+(\sum_{k=1}^{n}\sum_{u\in S_k-E_k}\psi_k(u))$}.$$
\end{theorem}

\begin{proof}
Consider $a\in(\prod_{i=1}^{n}S_i)-(\prod_{i=1}^{n}E_i)$ and $\gamma\in \mathbb{N}^{n}$ with $\sum_{i=1}^{n}\left\lfloor\frac{\gamma_i}{\psi_i(a_i)}\right\rfloor\leqslant t-1$. We claim that
$$\mbox{$(\prod_{\lambda\in L}\rho_{\lambda})(x_1+a_1,\dots,x_n+a_n)_{[\gamma]}=0$}.$$
Indeed, suppose that this does not hold. Then, we can choose $(\theta(\lambda)\mid\lambda\in L)$ such that $\gamma=\sum_{\lambda\in L}\theta(\lambda)$, where $\theta(\lambda)\in\supp(\rho_{\lambda}(x_1+a_1,\dots,x_n+a_n))$ for all $\lambda\in L$. For any $\lambda\in L$ with $\rho_{\lambda}(a_1,\dots,a_n)=0$, noticing that $\rho_{\lambda}(x_1+a_1,\dots,x_n+a_n)_{[\mathbf{0}]}=0$, we have $\theta(\lambda)\neq\mathbf{0}$. It then follows that
$$\mbox{$\sum_{i=1}^{n}\gamma_i\geqslant|\{\lambda\in L\mid\rho_{\lambda}(a_1,\dots,a_n)=0)\}|=|\{\lambda\in L\mid a\in H_\lambda\}|$},$$
a contradiction to (i), as desired. Now an application of Theorem 5.3 completes the proof.
\end{proof}

\setlength{\parindent}{2em}
Now we consider the special case that $R$ is an integral domain.

\setlength{\parindent}{0em}
\begin{corollary}
Suppose that $R$ is an integral domain. Let $L$ be a finite set, and let $(H_\lambda\mid\lambda\in L)$ be a tuple of subsets of $R^{n}$ with $\prod_{i=1}^{n}S_i\nsubseteq\bigcup_{\lambda\in L}H_\lambda$. Fix $t\in \mathbb{Z}^{+}$. Assume in addition that for any $a\in(\prod_{i=1}^{n}S_i)-(\prod_{i=1}^{n}E_i)$,
$$\mbox{$|\{\lambda\in L\mid a\in H_\lambda\}|\geqslant\max\{\sum_{i=1}^{n}\beta_i\mid\beta\in \mathbb{N}^{n},\sum_{i=1}^{n}\left\lfloor\frac{\beta_i}{\psi_i(a_i)}\right\rfloor\leqslant t-1\}+1$}.$$
Then, the following two statements hold true:

{\bf{(1)}}\,\,Let $(e_\lambda\mid\lambda\in L)\in\mathbb{N}^{L}$, and suppose that for any $\lambda\in L$, $H_\lambda$ is an $e_\lambda$-hyperplane of $R^{n}$. Then, for any $m\in[1,n]$, it holds that
$$\mbox{$\sum_{\lambda\in L}e_{\lambda}\geqslant(t-1)(\sum_{u\in S_m-E_m}\psi_m(u))+(\sum_{k=1}^{n}\sum_{u\in S_k-E_k}\psi_k(u))$};$$

{\bf{(2)}}\,\,Suppose that for any $\lambda\in L$, $H_\lambda$ is a hyperplane of $R^{n}$. Then, for any $m\in[1,n]$, it holds that
$$\mbox{$|L|\geqslant(t-1)(\sum_{u\in S_m-E_m}\psi_m(u))+(\sum_{k=1}^{n}\sum_{u\in S_k-E_k}\psi_k(u))$}.$$
\end{corollary}

\begin{proof}
{\bf{(1)}}\,\,By Definition 2.5, we can choose $(\rho_{\lambda}\mid\lambda\in L)\in\Omega^{L}$ such that for any $\lambda\in L$, it holds that $H_\lambda=\{u\in R^{n}\mid \rho_\lambda(u_1,\dots,u_n)=0\}$ and $\deg(\rho_\lambda)=e_\lambda$. Now we choose $w\in\prod_{i=1}^{n}S_i$ with $w\not\in\bigcup_{\lambda\in L}H_\lambda$. For any $\lambda\in L$, by $w\not\in H_\lambda$, we have $\rho_\lambda(w_1,\dots,w_n)\neq0$. Noticing that $R$ is an integral domain, we further derive that $\prod_{\lambda\in L}\rho_\lambda(w_1,\dots,w_n)\neq0$. Now the desired result immediately follows from Theorem 6.1.

{\bf{(2)}}\,\,This immediately follows from (1).
\end{proof}

\setlength{\parindent}{2em}
If $R$ is a field, then we have the following consequence of Corollary 6.1.

\setlength{\parindent}{0em}
\begin{corollary}
Suppose that $R$ is a field. Let $\langle~,~\rangle:R^{n}\times R^{n}\longrightarrow R$ be a non-degenerate bilinear map. Let $Y$ be a finite set, $(\eta_\lambda\mid\lambda\in Y)\in(R^{n})^{Y}$, $(c_\lambda\mid\lambda\in Y)\in R^{Y}$. Fix $t\in \mathbb{Z}^{+}$. Assume in addition that the following two conditions hold:

{\bf{(i)}}\,\,For any $a\in(\prod_{i=1}^{n}S_i)-(\prod_{i=1}^{n}E_i)$, it holds that
$$\mbox{$|\{\lambda\in Y\mid \langle\eta_\lambda,a\rangle=c_\lambda\}|\geqslant\max\{\sum_{i=1}^{n}\beta_i\mid\beta\in \mathbb{N}^{n},\sum_{i=1}^{n}\left\lfloor\frac{\beta_i}{\psi_i(a_i)}\right\rfloor\leqslant t-1\}+1$};$$
{\bf{(ii)}}\,\,There exists $w\in\prod_{i=1}^{n}S_i$ such that $\langle\eta_\lambda,w\rangle\neq c_\lambda$ for all $\lambda\in Y$.

Then, for any $m\in[1,n]$, it holds that
$$\mbox{$|\{\lambda\in Y\mid\eta_\lambda\neq0\}|\geqslant(t-1)(\sum_{u\in S_m-E_m}\psi_m(u))+(\sum_{k=1}^{n}\sum_{u\in S_k-E_k}\psi_k(u))$}.$$
\end{corollary}

\begin{proof}
Let $L=\{\lambda\in Y \mid\eta_\lambda\neq0\}$. Moreover, for any $\lambda\in L$, we let $H_\lambda=\{u\in R^{n}\mid \langle\eta_\lambda,u\rangle=c_\lambda\}$. Then, $(H_\lambda\mid\lambda\in L)$ is a tuple of hyperplanes of $R^{n}$. One can check from (ii) that $\prod_{i=1}^{n}S_i\nsubseteq\bigcup_{\lambda\in L}H_\lambda$; and moreover, for any $a\in R^{n}$, it holds that $\{\lambda\in Y\mid \langle\eta_\lambda,a\rangle=c_\lambda\}=\{\lambda\in L\mid a\in H_\lambda\}$. Hence an application of (2) of Corollary 6.1 implies the desired result.
\end{proof}

\setlength{\parindent}{0em}
\begin{remark}
{\bf{(1)}}\,\,In Theorem 6.1, we use the weaker assumption that all the $S_i$'s satisfy Condition (D) instead of that $R$ is an integral domain. This is inspired by [23, Theorem 10] which considers hyperplanes covering the Boolean cube $\{0,1_R\}^{n}$. We note that Theorem 6.1 includes [23, Theorem 10] as a special case.

{\bf{(2)}}\,\,Suppose that $R$ is a field. If $\psi_i(u)=1$ for all $i\in[1,n]$ and $u\in S_i$, then Corollary 6.1 recovers [9, Theorem 5.3]; and if $t=1$ and $\prod_{i=1}^{n}E_i=\{0\}$, then Corollary 6.1 recovers [23, Theorem 9] and [24, Theorem 12].

{\bf{(3)}}\,\,Suppose that $R$ is a finite field with $|R|=q$. Fix $t\in \mathbb{Z}^{+}$. Let $Y$ be a finite set, and let $(\eta_\lambda\mid\lambda\in Y)\in(R^{n})^{Y}$ such that $|\{\lambda\in Y\mid \eta_\lambda\in H\}|\geqslant t$ for any hyperplane $H\subseteq R^{n}$. Then, Corollary 6.2 implies that
\begin{equation}|Y|\geqslant(n+t-1)(q-1)+1,\end{equation}
which recovers Bruen's result [15, Theorem 2.1] (also see [9, Theorem 5.2]). (6.1) was first established for $t=1$ by Jamison, and then independently by Brouwer and Schrijver via different methods (see \cite{22,13}). We also refer the reader to \cite{7,8,26} for improvements and relevant work of (6.1).
\end{remark}

\subsection{The generalized Alon-F\"{u}redi theorem}

\setlength{\parindent}{2em}
In [4, Theorem 5], Alon and F\"{u}redi establish a general lower bound for the number of nonzero points of a polynomial $f\in\mathbb{F}[x_1,\dots,x_n]$ in $\prod_{i=1}^{n}S_i$, where $\mathbb{F}$ is a field and the $S_i$'s are finite subsets of $\mathbb{F}$. In [9, Corollary 4.3], Ball and Serra recover the Alon-F\"{u}redi theorem by using their punctured Nullstellensatz. In [18, Theorem 2.2], Clark generalizes the Alon-F\"{u}redi theorem to polynomials over a commutative ring by using a variant of Ball and Serra's punctured Nullstellensatz. In [12, Theorem 1.2], Bishnoi, Clark, Potukuchi and Schmitt further generalize [18, Theorem 2.2]. Their result is referred to as the generalized Alon-F\"{u}redi theorem, and has many applications in combinatorics and coding theory (see [12, Sections 4--6] for more details).

In this section, we give an alternative proof of the generalized Alon-F\"{u}redi theorem by using Theorems 5.2 and 5.3. Our proof closely follows the spirits in the proofs of [9, Corollary 4.3] and [18, Theorem 2.2].

\setlength{\parindent}{0em}
\begin{theorem}(Generalized Alon-F\"{u}redi Theorem, [12, Theorem 1.2])
Let $S_1,\dots,S_n$ be nonempty finite subsets of $R$ such that for any $k\in[1,n]$, $S_{k}$ satisfies Condition (D) in $R$. Let $\beta\in\mathbb{N}^{n}$ with $\beta\leqslant(|S_1|-1,\dots,|S_n|-1)$, and let $f\in\Omega-\{0\}$ such that $\alpha\leqslant\beta$ for all $\alpha\in\supp(f)$. Then, there exists $\mu\in\mathbb{N}^{n}$ satisfying the following three conditions:

{\bf{(1)}}\,\,$(|S_1|-\beta_1,\dots,|S_n|-\beta_n)\leqslant\mu\leqslant(|S_1|,\dots,|S_n|)$;

{\bf{(2)}}\,\,$\sum_{k=1}^{n}\mu_k=(\sum_{k=1}^{n}|S_k|)-\deg(f)$;

{\bf{(3)}}\,\,$|\{v\in\prod_{i=1}^{n}S_i\mid f(v_1,\dots,v_n)\neq0\}|\geqslant\prod_{k=1}^{n}\mu_k$.
\end{theorem}

\begin{proof}
Let $Q=\{v\in\prod_{i=1}^{n}S_i\mid f(v_1,\dots,v_n)\neq0\}$. Moreover, for any $k\in[1,n]$, let $E_k=\{v_k\mid v\in Q\}$, and let $\psi_k$ denote the constant $1$ map defined on $S_k$. Then, we have $f(w_1,\dots,w_n)=0$ for all $w\in(\prod_{i=1}^{n}S_i)-(\prod_{i=1}^{n}E_i)$. Since $f\neq0$ and $\alpha\leqslant(|S_1|-1,\dots,|S_n|-1)$ for all $\alpha\in\supp(f)$, an application of (2) of Theorem 5.2 and Theorem 5.3 to $(S_1,\psi_1),\dots,(S_n,\psi_n)$ and $t=1$ leads to the facts that $Q\neq\emptyset$ and
\begin{equation}\mbox{$\prod_{k=1}^{n}\prod_{u\in S_k-E_k}(x_k-u)\mid f$}.\end{equation}
If $\beta=\mathbf{0}$, then $(|S_1|,\dots,|S_n|)$ satisfies Conditions (1)--(3), as desired. Therefore in what follows, we assume that $\beta\neq\mathbf{0}$. Hence we can choose $m\in[1,n]$ with $\beta_m\geqslant1$. Since $Q\neq\emptyset$, we can choose $b\in E_m$ such that
\begin{equation}|Q|\geqslant|E_m|\cdot|\{v\in Q\mid v_m=b\}|.\end{equation}
Let $g\triangleq f(x_1,\dots,x_{m-1},b,x_{m+1},\dots,x_n)$. It is straightforward to verify that
\begin{equation}\mbox{$|\{v\in\prod_{i=1}^{n}S_i\mid g(v_1,\dots,v_n)\neq0\}|=|S_m|\cdot|\{v\in Q\mid v_m=b\}|$}.\end{equation}
By $b\in E_m$ and (6.4), we have $g\neq0$. Moreover, (6.3) and (6.4) imply that
\begin{equation}\mbox{$|Q|\geqslant\frac{|E_m|}{|S_m|}\cdot|\{v\in\prod_{i=1}^{n}S_i\mid g(v_1,\dots,v_n)\neq0\}|$}.\end{equation}
Let $\theta=(\beta_1,\dots,\beta_{m-1},0,\beta_{m+1},\dots,\beta_n)\in\mathbb{N}^{n}$. Then, it is straightforward to verify that $\gamma\leqslant\theta$ for all $\gamma\in\supp(g)$. Since $\theta\leqslant\beta$, $\theta\neq\beta$, by induction, we can choose $\varepsilon\in\mathbb{N}^{n}$ satisfying the following three conditions:

$(i)$\,\,$(|S_1|-\theta_1,\dots,|S_n|-\theta_n)\leqslant\varepsilon\leqslant(|S_1|,\dots,|S_n|)$;

$(ii)$\,\,$\sum_{k=1}^{n}\varepsilon_k=(\sum_{k=1}^{n}|S_k|)-\deg(g)$;

$(iii)$\,\,$|\{v\in\prod_{i=1}^{n}S_i\mid g(v_1,\dots,v_n)\neq0\}|\geqslant\prod_{k=1}^{n}\varepsilon_k$.

We note that from $\theta_m=0$ and $(i)$, we have $\varepsilon_m=|S_m|$. Now let
\begin{equation}\lambda=(\varepsilon_1,\dots,\varepsilon_{m-1},|E_m|,\varepsilon_{m+1},\dots,\varepsilon_n)\in\mathbb{N}^{n}.\end{equation}
We will show that $\lambda$ satisfies the following three conditions:

$(iv)$\,\,$(|S_1|-\beta_1,\dots,|S_n|-\beta_n)\leqslant\lambda\leqslant(|S_1|,\dots,|S_n|)$;

$(v)$\,\,$\sum_{k=1}^{n}\lambda_k\geqslant(\sum_{k=1}^{n}|S_k|)-\deg(f)$;

$(vi)$\,\,$|Q|\geqslant\prod_{k=1}^{n}\lambda_k$.

Indeed, for any $k\in[1,n]-\{m\}$, by $\beta_k=\theta_k$, $\lambda_k=\varepsilon_k$ and $(i)$, we have $|S_k|-\beta_k\leqslant\lambda_k\leqslant|S_k|$. From (6.2), $f\neq0$ and (2) of Lemma 2.1, we can choose $\alpha\in\supp(f)$ with $|S_m|-|E_m|\leqslant\alpha_m$. Since $\alpha\leqslant\beta$, we have $|S_m|-|E_m|\leqslant\beta_m$, which, together with $\lambda_m=|E_m|$, implies that $|S_m|-\beta_m\leqslant\lambda_m\leqslant|S_m|$, which further establishes $(iv)$. Next, by (6.2), we can choose $h\in\Omega$ with $f=(\prod_{u\in S_m-E_m}(x_m-u))\cdot h$. It then follows that
$$\mbox{$g=(\prod_{u\in S_m-E_m}(b-u))\cdot h(x_1,\dots,x_{m-1},b,x_{m+1},\dots,x_n)$},$$
which, together with (4) of Lemma 2.1, further implies that
\begin{equation}\deg(g)\leqslant\deg(h)=\deg(f)-|S_m|+|E_m|.\end{equation}
With (6.6), $\varepsilon_m=|S_m|$ and some straightforward computation, $(v)$ follows from $(ii)$ and (6.7), and $(vi)$ follows from $(iii)$ and (6.5), as desired. Noticing that $\alpha\leqslant\beta$ for all $\alpha\in\supp(f)$, we have $\deg(f)\leqslant\sum_{k=1}^{n}\beta_k$, which, together with $(v)$, implies that
$$\mbox{$\sum_{k=1}^{n}(|S_k|-\beta_k)\leqslant(\sum_{k=1}^{n}|S_k|)-\deg(f)\leqslant\sum_{k=1}^{n}\lambda_k$}.$$
Since $(|S_1|-\beta_1,\dots,|S_n|-\beta_n)\leqslant\lambda$ (see $(iv)$), we can choose $\mu\in\mathbb{N}^{n}$ with $(|S_1|-\beta_1,\dots,|S_n|-\beta_n)\leqslant\mu\leqslant\lambda$ and $\sum_{k=1}^{n}\mu_k=(\sum_{k=1}^{n}|S_k|)-\deg(f)$. By $(iv)$ and $(vi)$, we conclude that $\mu$ satisfies Conditions (1)--(3), as desired.
\end{proof}

\setlength{\parindent}{0em}
\begin{remark}
It has been proven in [12, Theorem 1.2] that the lower bound is sharp in all cases. We refer the reader to [12, Section 3.3] for more details.
\end{remark}

\section*{Appendix}\appendix

\section{Proof of Lemma 2.1}

\setlength{\parindent}{0em}
{\bf{(1)}}\,\,Let $\gamma\in\max(\supp(f))$. Then, we have
$$(f\cdot g)_{[\gamma+\theta]}=\sum_{(\alpha\in\supp(f),\beta\in\supp(g),\alpha+\beta=\gamma+\theta)}f_{[\alpha]} g_{[\beta]}.$$
For $\alpha\in\supp(f)$, $\beta\in\supp(g)$ with $\alpha+\beta=\gamma+\theta$, since $\theta$ is the greatest element of $\supp(g)$, we have $\beta\leqslant\theta$, and hence $\gamma\leqslant\alpha$, which, together with $\gamma\in\max(\supp(f))$, $\alpha\in\supp(f)$, implies that $\alpha=\gamma$, which further implies that $\beta=\theta$. Therefore we have $(f\cdot g)_{[\gamma+\theta]}=f_{[\gamma]} g_{[\theta]}=f_{[\gamma]}$, as desired.

{\bf{(2)}} and {\bf{(3)}}\,\,We begin by noting that for $A\subseteq B\subseteq\mathbb{N}^{n}$ such that $B$ is finite and $\max(B)\subseteq A$, it holds that $\max(A)=\max(B)$, $\Delta(A)=\Delta(B)$. Moreover, for any $C,D\subseteq\mathbb{N}^{n}$, it holds that $\Delta(C+D)=\Delta(C)+\Delta(D)$. Now we have $\supp(f\cdot g)\subseteq\supp(f)+\supp(g)\subseteq\mathbb{N}^{n}$ and $\supp(f)+\supp(g)$ is finite. Since $\theta$ is the greatest element of $\supp(g)$, we have
$$\max(\supp(f)+\supp(g))=\max(\supp(f))+\{\theta\}.$$
From (1), we deduce that $\max(\supp(f)+\supp(g))\subseteq\supp(f\cdot g)$, which further establishes the desired result.

{\bf{(4)}} and {\bf{(5)}}\,\,Note that (4) follows from (2) and some straightforward verification, and so we only prove (5). From (3), we deduce that
\begin{eqnarray*}
\begin{split}
A+\supp(h\cdot g)&\subseteq A+\supp(h)+\supp(g)\subseteq \Delta(\supp(f))+\supp(g)\\
&\subseteq\Delta(\supp(f))+\Delta(\supp(g))=\Delta(\supp(f\cdot g)),
\end{split}
\end{eqnarray*}
as desired.

\section{Proof of Lemma 2.4}
\setlength{\parindent}{2em}
The following notation will be used in this section and Appendix D.

\begin{notation}
For any $\alpha,\beta\in\mathbb{N}^{n}$, let $\alpha\wedge\beta\in\mathbb{N}^{n}$ be defined as $(\alpha\wedge\beta)_i=\min\{\alpha_i,\beta_i\}$ for all $i\in[1,n]$.
\end{notation}

\setlength{\parindent}{0em}
\begin{lemma}
Let $g$ be a monic polynomial, and let $\theta$ be the greatest element of $\supp(g)$. Also let $\rho\in\Omega$ such that $\gamma\wedge\theta=\mathbf{0}$ for all $\gamma\in\supp(\rho)$. Let $\psi\in\Omega$ such that $\rho\mid\psi$, $g\mid\psi$. Then, we have $(\rho\cdot g)\mid\psi$.
\end{lemma}

\begin{proof}
Since $\rho\mid\psi$, we can choose $f\in\Omega$ with $\psi=f\cdot \rho$. By Lemma 3.1, we can choose $\tau\in\Omega$ such that $g\mid(f-\tau)$ and $\theta\nleqslant\alpha$ for all $\alpha\in\supp(\tau)$. It then follows that $(\rho\cdot g)\mid(\psi-\rho\cdot\tau)$. We claim that $\theta\nleqslant\beta$ for all $\beta\in\supp(\rho\cdot \tau)$. Indeed, let $\beta\in\supp(\rho\cdot \tau)$. Then, we can choose $\alpha\in\supp(\tau)$, $\gamma\in\supp(\rho)$ with $\beta=\alpha+\gamma$. By $\alpha\in\supp(\tau)$, we have $\theta\nleqslant\alpha$, and by $\gamma\in\supp(\rho)$, we have $\gamma\wedge\theta=\mathbf{0}$. It then follows that $\theta\nleqslant\alpha+\gamma=\beta$, as desired. Now by $g\mid\psi$, $g\mid(\psi-\rho\cdot\tau)$, we have $g\mid(\rho\cdot\tau)$. It then follows from (2) of Lemma 2.1 that $\max(\supp(\rho\cdot\tau))=\emptyset$, and hence $\rho\cdot\tau=0$. Therefore from $(\rho\cdot g)\mid(\psi-\rho\cdot\tau)$, we deduce that $(\rho\cdot g)\mid\psi$, as desired.
\end{proof}

\setlength{\parindent}{2em}
We are now ready to prove Lemma 2.4.
\begin{proof}[Proof of Lemma 2.4]
Immediately follows from Lemma B.1 and an induction argument.
\end{proof}

\section{Proofs of Lemmas 2.5 and 2.6}

\setlength{\parindent}{0em}
\begin{proof}[Proof of Lemma 2.5]
If $\alpha\not\in(\mathbb{Z}^{+})^{n}$, then one can check that $\mathbf{0}\in B$, and hence $|\mathbb{N}^{n}-\nabla(B)|=0=\left(\prod_{i=1}^{n}\alpha_i\right)\cdot\binom{n+t-1}{n}$, as desired. Therefore in what follows, we assume that $\alpha\in(\mathbb{Z}^{+})^{n}$. Then, for an arbitrary $\beta\in\mathbb{N}^{n}$, we have
\begin{eqnarray*}
\begin{split}
\beta\in\mathbb{N}^{n}-\nabla(B)&\Longleftrightarrow(\mbox{$\forall~\theta\in\mathbb{N}^{n}~s.t.~\sum_{i=1}^{n}\theta_i=t:(\alpha_1\theta_1,\dots,\alpha_n\theta_n)\nleqslant\beta$})\\
&\Longleftrightarrow(\mbox{$\forall~\theta\in\mathbb{N}^{n}~s.t.~\sum_{i=1}^{n}\theta_i=t:\theta\nleqslant\left(\left\lfloor\frac{\beta_1}{\alpha_1}\right\rfloor,\dots,\left\lfloor\frac{\beta_n}{\alpha_n}\right\rfloor\right)$})\\
&\Longleftrightarrow\mbox{$\sum_{i=1}^{n}\left\lfloor\frac{\beta_i}{\alpha_i}\right\rfloor\leqslant t-1$},
\end{split}
\end{eqnarray*}
which establishes (2.13). Now, with the following two equations:
$$\mbox{$|\{\gamma\in\mathbb{N}^{n}\mid\sum_{i=1}^{n}\gamma_i\leqslant t-1\}|=\binom{n+t-1}{n}$},$$
$$\mbox{$\forall~\gamma\in\mathbb{N}^{n}:|\{\beta\in\mathbb{N}^{n}\mid\left(\left\lfloor\frac{\beta_1}{\alpha_1}\right\rfloor,\dots,\left\lfloor\frac{\beta_n}{\alpha_n}\right\rfloor\right)=\gamma\}|=\prod_{i=1}^{n}\alpha_i$},$$
(2.12) immediately follows from (2.13), as desired.
\end{proof}

\setlength{\parindent}{0em}
\begin{proof}[Proof of Lemma 2.6]
Let $\mathbf{1}\in\mathbb{N}^{n}$ denote the all $1$ vector. Then, via some straightforward verification, we have
$$\hspace*{-12mm}\mbox{$\nabla(C)-\nabla(B)=\{(\alpha_1\theta_1,\dots,\alpha_n\theta_n)+\tau\mid\theta\in\mathbb{N}^{n},\sum_{i=1}^{n}\theta_i=t-1,\tau\in\mathbb{N}^{n},\alpha-\gamma\leqslant\tau\leqslant\alpha-\mathbf{1}\}$}.$$
Moreover, for any $\theta,\tau,\beta,\eta\in\mathbb{N}^{n}$ with $\tau\leqslant\alpha-\mathbf{1}$, $\eta\leqslant\alpha-\mathbf{1}$, it holds that
$$(\alpha_1\theta_1,\dots,\alpha_n\theta_n)+\tau=(\alpha_1\beta_1,\dots,\alpha_n\beta_n)+\eta\Longrightarrow\theta=\beta,\tau=\eta.$$
By the above discussion, we have
\begin{eqnarray*}
\begin{split}
|\nabla(C)-\nabla(B)|&=\mbox{$|\{\tau\in\mathbb{N}^{n}\mid\alpha-\gamma\leqslant\tau\leqslant\alpha-\mathbf{1}\}|\cdot|\{\theta\in\mathbb{N}^{n}\mid\sum_{i=1}^{n}\theta_i=t-1\}|$}\\
&=\mbox{$(\prod_{i=1}^{n}\gamma_i)\cdot\binom{n+t-2}{n-1}$}.
\end{split}
\end{eqnarray*}
It then follows from Lemma 2.5 that
\begin{eqnarray*}
\begin{split}
|\mathbb{N}^{n}-\nabla(B\cup C)|&=|\mathbb{N}^{n}-\nabla(B)|-|\nabla(C)-\nabla(B)|\\
&=\mbox{$(\prod_{i=1}^{n}\alpha_i)\binom{n+t-1}{n}-(\prod_{i=1}^{n}\gamma_i)\binom{n+t-2}{n-1}$},
\end{split}
\end{eqnarray*}
as desired.
\end{proof}

\section{Establishing (2) of Lemma 2.7}
\setlength{\parindent}{2em}
We begin with some remarks on the notion of $\mathbf{S}$-polynomial (see [1, Definition 1.7.1], [16, Section 2.9], [21, Definition 21.29]). Suppose that $f,g$ are monic polynomials, and $\alpha,\beta$ are the greatest elements of $\supp(f)$ and $\supp(g)$, respectively. Then, the $\mathbf{S}$-polynomial of $(f,g)$, denoted by $\mathbf{S}(f,g)$, is defined as
\begin{equation}\mbox{$\mathbf{S}(f,g)\triangleq\left(\prod_{i=1}^{n}{x_i}^{\beta_i-(\alpha\wedge\beta)_i}\right)\cdot f-\left(\prod_{i=1}^{n}{x_i}^{\alpha_i-(\alpha\wedge\beta)_i}\right)\cdot g$}.\end{equation}
In general, $\mathbf{S}$-polynomials are defined with respect to a fixed monomial order, and for arbitrary polynomials. For monic polynomials, (D.1) is indeed a special case of the general definition of $\mathbf{S}$-polynomial.

The celebrated Buchberger's Theorem characterizes Gr\"{o}bner bases in terms of $\mathbf{S}$-polynomials (see [1, Theorem 1.7.4], [16, Section 2.10], [21, Theorem 21.31]). Buchberger's Theorem is established for polynomials over a field. Due to the fact that any monic polynomial has leading coefficient $1_R$ (see Remark 2.1), one can check that the proof of Buchberger's Theorem remains valid in our setting, which further leads to the following lemma.

\setlength{\parindent}{0em}
\begin{lemma}
Let $\Lambda$ be a finite set, and let $(g(\lambda)\mid\lambda\in\Lambda)\in\Omega^{\Lambda}$ be a family of monic polynomials. Assume that for any $(\xi,\mu)\in\Lambda\times\Lambda$, there exists $(p(\lambda)\mid\lambda\in\Lambda)\in\Omega^{\Lambda}$ such that $\mathbf{S}(g(\xi),g(\mu))=\sum_{\lambda\in\Lambda}p(\lambda)\cdot g(\lambda)$ and
$$\forall~\lambda\in\Lambda:\supp(p(\lambda))+\supp(g(\lambda))\subseteq\Delta(\supp(\mathbf{S}(g(\xi),g(\mu)))).$$
Then, $(g(\lambda)\mid\lambda\in\Lambda)$ is a Gr\"{o}bner basis of $\langle\{g(\lambda)\mid\lambda\in\Lambda\}\rangle$.
\end{lemma}

\setlength{\parindent}{0em}
\begin{lemma}
Let $(g_1,\dots,g_n)\in\Omega^{n}$ be a family of monic polynomials such that $g_k\in R[x_{k}]$. Fix $\alpha\in\mathbb{N}^{n}$, $\beta\in\mathbb{N}^{n}$. Let $\varphi=\mathbf{S}(\prod_{k=1}^{n}{g_{k}}^{\alpha_{k}},\prod_{k=1}^{n}{g_{k}}^{\beta_{k}})$, and let
\begin{equation}\mbox{$h_1=(\prod_{k=1}^{n}{x_{k}}^{c_{k}(\beta_{k}-(\alpha\wedge\beta)_k)})-(\prod_{k=1}^{n}{g_{k}}^{\beta_{k}-(\alpha\wedge\beta)_k})$},\end{equation}
\begin{equation}\mbox{$h_2=(\prod_{k=1}^{n}{g_{k}}^{\alpha_{k}-(\alpha\wedge\beta)_k})-(\prod_{k=1}^{n}{x_{k}}^{c_{k}(\alpha_{k}-(\alpha\wedge\beta)_k)})$},\end{equation}
where $c_k\triangleq\deg(g_k)$ for all $k\in[1,n]$. Then, the following three equations hold:
\begin{equation}\mbox{$\varphi=h_1\cdot(\prod_{k=1}^{n}{g_{k}}^{\alpha_{k}})+h_2\cdot(\prod_{k=1}^{n}{g_{k}}^{\beta_{k}})$},\end{equation}
\begin{equation}\mbox{$\supp(h_1)+\supp(\prod_{k=1}^{n}{g_{k}}^{\alpha_{k}})\subseteq\Delta(\supp(\varphi))$},\end{equation}
\begin{equation}\mbox{$\supp(h_2)+\supp(\prod_{k=1}^{n}{g_{k}}^{\beta_{k}})\subseteq\Delta(\supp(\varphi))$}.\end{equation}
\end{lemma}

\begin{proof}
Let $\lambda=\beta-(\alpha\wedge\beta)$, $\mu=\alpha-(\alpha\wedge\beta)$, $\theta=(c_{1}\lambda_{1},\dots,c_{n}\lambda_{n})$, $\rho=(c_{1}\mu_{1},\dots,c_{n}\mu_{n})$, $f=\mathbf{S}(\prod_{k=1}^{n}{g_{k}}^{\mu_{k}},\prod_{k=1}^{n}{g_{k}}^{\lambda_{k}})$. Since $\lambda\wedge\mu=\mathbf{0}$, we have $\theta\wedge\rho=\mathbf{0}$. Therefore by (D.1), we have
\begin{equation}\mbox{$f=(\prod_{k=1}^{n}{x_{k}}^{\theta_{k}})(\prod_{k=1}^{n}{g_{k}}^{\mu_{k}})-(\prod_{k=1}^{n}{x_{k}}^{\rho_{k}})(\prod_{k=1}^{n}{g_{k}}^{\lambda_{k}})$}.\end{equation}
Now we show that the following three equations hold:
\begin{equation}\mbox{$\varphi=f\cdot(\prod_{k=1}^{n}{g_{k}}^{(\alpha\wedge\beta)_k})=h_1\cdot(\prod_{k=1}^{n}{g_{k}}^{\alpha_{k}})+h_2\cdot(\prod_{k=1}^{n}{g_{k}}^{\beta_{k}})$},\end{equation}
\begin{equation}\mbox{$\supp(h_1)+\supp(\prod_{k=1}^{n}{g_{k}}^{\mu_k})\subseteq\Delta(\supp(f))$},\end{equation}
\begin{equation}\mbox{$\supp(h_2)+\supp(\prod_{k=1}^{n}{g_{k}}^{\lambda_k})\subseteq\Delta(\supp(f))$}.\end{equation}
First, (D.8) follows from some straightforward computation. Second, we prove (D.9). Let $\varepsilon\in\supp(h_1)$. From (D.2), we deduce that $\varepsilon\leqslant\theta$, $\varepsilon\neq\theta$, which, together with $\theta\wedge\rho=\mathbf{0}$, implies that $\theta\nleqslant\varepsilon+\rho$, and hence $\varepsilon+\rho\not\in\supp((\prod_{k=1}^{n}{x_{k}}^{\theta_{k}})(\prod_{k=1}^{n}{g_{k}}^{\mu_{k}}))$. Also noticing that $\varepsilon\in\supp(\prod_{k=1}^{n}{g_{k}}^{\lambda_{k}})$, we have $\varepsilon+\rho\in\supp((\prod_{k=1}^{n}{x_{k}}^{\rho_{k}})(\prod_{k=1}^{n}{g_{k}}^{\lambda_{k}}))$. It then follows from (D.7) that $\varepsilon+\rho\in\supp(f)$. By the arbitrariness of $\varepsilon$, we have $\supp(h_1)+\{\rho\}\subseteq\supp(f)$, which, together with the fact that $\rho$ is the greatest element of $\supp(\prod_{k=1}^{n}{g_{k}}^{\mu_{k}})$, further establishes (D.9). Third, with the help of (D.3) and (D.7), (D.10) can be established in a parallel fashion. Now (D.4) follows from (D.8). Moreover, with (D.8) and (5) of Lemma 2.1, (D.5) and (D.6) follow from (D.9) and (D.10), respectively.
\end{proof}

\setlength{\parindent}{2em}
Finally, we conclude that (2) of Lemma 2.7 immediately follows from Lemmas D.1 and D.2, as desired.

\end{document}